\documentclass[12pt]{amsart}
\usepackage[T2A]{fontenc}
\usepackage[english]{babel}
\usepackage{amsaddr}
\usepackage{amsmath}
\usepackage{amssymb}
\usepackage{amsfonts}
\usepackage{epsfig}
\usepackage{srcltx}
\usepackage{subfigure}
\usepackage{cite}
\usepackage{float}
\usepackage{mathtools}
\usepackage[a4paper, mag=1000, includefoot, left=2cm, right=2cm, top=2cm, bottom=2cm, headsep=1cm, footskip=1cm]{geometry}
  \usepackage[unicode,colorlinks]{hyperref}
   \hypersetup{colorlinks=true, citecolor=blue, linkcolor=blue}

\newtheorem{Th}{Theorem}
\newtheorem{Lem}{Lemma}

\begin{document}

\thispagestyle{empty}

\title[Nonlinear resonance in oscillatory systems with decaying perturbations]{Nonlinear resonance in oscillatory systems with decaying perturbations}

\author[O.A. Sultanov]{Oskar A. Sultanov}

\address{
%Chebyshev Laboratory, St. Petersburg State University, 14th Line V.O., 29, Saint Petersburg 199178 Russia;\\
Institute of Mathematics, Ufa Federal Research Centre, Russian Academy of Sciences, Chernyshevsky street, 112, Ufa 450008 Russia.}
\email{oasultanov@gmail.com}

%\thanks{\it \today}

\maketitle

{\small
\begin{quote}
\noindent{\bf Abstract.} 
Time-decaying perturbations of nonlinear oscillatory systems in the plane are considered. It is assumed that the unperturbed systems are non-isochronous and the perturbations oscillate with an asymptotically constant frequency. Resonance effects and long-term asymptotic regimes for solutions are investigated. In particular, the emergence of stable states close to periodic ones is discussed. By combining the averaging technique and stability analysis, the conditions on perturbations are described that guarantee the existence and stability of the phase-locking regime with a resonant amplitude. The results obtained are applied to the perturbed Duffing oscillator.
 \medskip


%	34C15  	Nonlinear oscillations and coupled oscillators for ordinary differential equations
%	34C29  	Averaging method for ordinary differential equations
%	34E10  	Perturbations, asymptotics of solutions to ordinary differential equations
% 34D20  	Stability of solutions to ordinary differential equations

\end{quote}
}
{\small

\section*{Introduction}
Perturbation of nonlinear oscillatory systems is a classical problem with a wide range of applications~\cite{GH83,AF06}. In this paper, time-decaying perturbations are considered and a class of asymptotically autonomous systems in the plane is investigated. Note that asymptotically autonomous systems arise, for example, when studying steady-state modes in multidimensional problems by reducing the dimension~\cite{CCT95,DS22}, when constructing the asymptotics of strongly nonlinear non-autonomous systems by isolating growing terms of solutions~\cite{BG08,KF13}, and in various problems with time-dependent damping~\cite{SFR19,JM23}.

Qualitative properties of asymptotically autonomous systems have been studied in many papers~\cite{LM56,LRS02,KS05,MR08,OS22Non}. In particular, it follows from~\cite{LDP74} that under certain conditions, time-decaying perturbations may not disturb the global dynamics of oscillatory systems. However, in the general case, the dynamics of perturbed and unperturbed systems can differ significantly~\cite{HT94,OS21IJBC}. 

This paper studies the effect of damped oscillatory perturbations with an asymptotically constant frequency on non-isochronous systems. Note that similar problems has been considered in several papers. In particular, the asymptotic analysis of linear systems with damped oscillatory perturbations was discussed in~\cite{HL75,MP85,PN07,BN10}. The asymptotic behaviour of solutions to nonlinear equations in the vicinity of the equilibrium was investigated in~\cite{OS23JMS}. Bifurcations in such systems related to the stability of the equilibrium were discussed in~\cite{OS21DCDS,OS21JMS}. To the best of the author’s knowledge, the influence of such perturbations on the behaviour of nonlinear systems far from equilibrium has not yet been discussed. This is the subject of the present paper. In particular, we study the emergence of near-periodic stable states due to resonance phenomena with damped oscillatory perturbations. Note also that similar effects in the problems with a small parameter are usually associated with nonlinear resonance and are considered to be well studied~\cite{BVC79,SUZ88,Sos97,JCRS11}. However, in this paper, the presence of a small parameter is not assumed. We discuss the role of time-decaying perturbations in the emergence and stability of long-term asymptotics regimes.

The paper is organized as follows. Section~\ref{sec1} provides the statement of the problem and a motivating example. The main results are presented in Section~\ref{sec2}. The justification is contained in subsequent sections. First, in Section~\ref{sec3}, we construct a near-identity transformation averaging the system in the first asymptotic terms. Section~\ref{sec4} analyses the truncated system obtained from the full system by dropping the remainder terms and describes possible asymptotic regimes. Section~\ref{sec5} discusses the persistence of these regimes in the full system by constructing Lyapunov functions. In Section~\ref{sEx}, the proposed theory is applied to examples of asymptotically autonomous systems. The paper concludes with a brief discussion of the results obtained.

\section{Problem statement}\label{sec1}
 Consider a system of two differential equations
\begin{gather}\label{PS}
\frac{dr}{dt}=f(r,\varphi,S(t),t), \quad \frac{d\varphi}{dt}=\omega(r)+g(r,\varphi,S(t),t),
\end{gather}
where the functions $\omega(r)>0$, $f(r,\varphi,S,t)$ and $g(r,\varphi,S,t)$ are infinitely differentiable, defined for all $|r|\leq \mathcal R={\hbox{\rm const}}$, $(\varphi,S)\in\mathbb R^2$, $t>0$, and are $2\pi$-periodic with respect to $\varphi$ and $S$, $\omega'(r)\not\equiv 0$. The functions $f(r,\varphi,S,t)$ and $g(r,\varphi,S,t)$ play the role of perturbations of the autonomous system 
\begin{gather}\label{US}
\frac{d\hat r}{dt}=0, \quad \frac{d\hat \varphi}{dt}=\omega(\hat r),
\end{gather}
describing non-isochronous oscillations on the plane $(x,y)=(\hat r\cos\hat \varphi,-\hat r\sin\hat \varphi)$ with a constant amplitude $\hat r(t)\equiv r_0$, $|r_0|< \mathcal R$ and a natural frequency $\omega(r_0)$. The solutions $r(t)$ and $\varphi(t)$ of system \eqref{PS} corresponds to the amplitude and the phase of the perturbed oscillations. 

It is assumed that the frequency of perturbations is asymptotically constant: $S'(t)\sim s_0$ as $t\to\infty$ with $s_0={\hbox{\rm const}}>0$, and the intensity decays with time: for each fixed $r$ and $\varphi$
\begin{gather*}
f(r,\varphi,S(t),t)\to 0, \quad g(r,\varphi,S(t),t)\to 0, \quad t\to\infty.
\end{gather*}
In this case, the perturbed system \eqref{PS} is asymptotically autonomous with the limiting system \eqref{US}. The goal of the paper is to study the resonant effects of perturbations $f(r,\varphi,S(t),t)$ and $g(r,\varphi,S(t),t)$ on the dynamics far from the equilibrium of the limiting system and to describe possible asymptotic regimes for solutions.

Let us specify the considered class of perturbations. We assume that 
\begin{gather}\label{fgas}\begin{split}
&f(r,\varphi,S,t)\sim  \sum_{j=1}^\infty t^{-\frac{j}{q}} f_j(r,\varphi,S), \\ 
&g(r,\varphi,S,t)\sim  \sum_{j=1}^\infty t^{-\frac{j}{q}} g_j(r,\varphi,S), \quad 
t\to\infty,
\end{split}
\end{gather}
for all $|r|< \mathcal R$ and $(\varphi,S)\in\mathbb R^2$, where $f_j(r,\varphi,S)$ and $g_j(r,\varphi,S)$ are $2\pi$-periodic with respect to $\varphi$ and $S$, and $q\in\mathbb Z_+$. 
The phase of perturbations is considered in the form
\begin{gather}\label{Sform}
S(t)=s_0 t + \sum_{j=1}^{q-1} s_j t^{1-\frac{j}{q}}+s_q \log t,
\end{gather}
where $s_j={\hbox{\rm const}}$. Moreover, it is assumed that there exist $0<|a|< \mathcal R$ and coprime integers $\kappa,\varkappa\in\mathbb Z_+$ such that the resonant condition holds:
\begin{gather}\label{rc}
	\kappa s_0=\varkappa\omega(a), \quad \eta:=\omega'(a)\neq 0.
\end{gather}
Note that the series in \eqref{fgas} are asymptotic as $t\to\infty$, and for all $N\geq 1$ the following estimates hold: $f(r,\varphi,S,t)-\sum_{j=0}^{N-1}t^{-j/q} f_j(r,\varphi,S)=\mathcal O(t^{-N/q})$ and $g(r,\varphi,S,t)-\sum_{j=0}^{N-1}t^{-j/q} g_j(r,\varphi,S)=\mathcal O(t^{-N/q})$ as $t\to\infty$ uniformly for all $|r|\leq \mathcal R$ and $(\varphi,S)\in\mathbb R^2$. Note also that instead of power functions one could consider another asymptotic scale, but in this case the calculations would be more complex and cumbersome.

Consider the example
\begin{gather}\label{Ex0}
\frac{dx}{dt}=y, \quad 
\frac{dy}{dt}=-x+\vartheta x^3+ t^{-\frac{1}{2}} Z(x,y,S(t)),
\end{gather}
where  $Z(x,y,S)\equiv \alpha(S)x+\beta(S)y$, $\alpha(S)\equiv \alpha_0+\alpha_1 \sin S$, $\beta(S)\equiv \beta_0+ \beta_1\sin S$, $S(t)\equiv 3 t/2$ with parameters $\alpha_i,\beta_i,\in\mathbb R$ and $\vartheta>0$. Let us show that this system corresponds to \eqref{PS}. The limiting system $d\hat x/dt=\hat y$, $d\hat y/dt=- U'(\hat x)$ with $U(x)\equiv x^2/2-\vartheta x^4/4$ has a stable equilibrium at $(0,0)$, and the level lines $\{(x,y)\in\mathbb R: U(x)+y^2/2=r^2/2\}$ for all $0<|r|< (2\vartheta)^{-1/2}$ correspond to $T(r)$-periodic solutions 
\begin{gather}
\nonumber 
\hat x_0(t,r)\equiv r\,{\hbox{\rm sn}}\left(\frac{t}{\sqrt{k_r^2+1}},k_r\right)\sqrt{k_r^2+1}, \quad 
\hat y_0(t,r)\equiv r\,{\hbox{\rm cn}}\left(\frac{t}{\sqrt{k_r^2+1}},k_r\right){\hbox{\rm dn}}\left(\frac{t}{\sqrt{k_r^2+1}},k_r\right), \\
\label{omegaeq} 
T(r)\equiv 4K(k_r)\sqrt{k_r^2+1}, \quad 
\omega(r)\equiv \frac{2\pi}{T(r)},
\end{gather}
where ${\hbox{\rm sn}}(u,k)$, ${\hbox{\rm cn}}(u,k)$, ${\hbox{\rm dn}}(u,k)$ are the Jacoby elliptic functions, $K(k)$ is the complete elliptic integral of the first kind, and $k_r\in (0,1)$ is a root of $(k_r+k_r^{-1})^{-2}=\vartheta  r^2/2$.
%2\int\limits_{x_-(r)}^{x_+(r)}\frac{dz}{\sqrt{r^2-z^2+\vartheta z^4/2}}, \quad 
%x_\pm(r)=\pm\sqrt{\frac{1-\sqrt{1-2\vartheta r^2}}{\vartheta}}.
Define auxiliary $2\pi$-periodic functions 
\begin{gather*}
X(\varphi,r)\equiv \hat x_0 \left(\frac{\varphi}{\omega(r)},r\right), \quad
Y(\varphi,r)\equiv \hat y_0 \left(\frac{\varphi}{\omega(r)},r\right).
\end{gather*}
It can easily be checked that $\omega(r)\partial_\varphi X=Y$, $\omega(r)\partial_\varphi Y=-U(X)$, $U(X)+Y^2/2=r^2/2$ and 
\begin{gather*}
{\hbox{\rm det}}\frac{\partial (X,Y)}{\partial (\varphi,r)}\equiv \begin{vmatrix} \partial_\varphi X & \partial_\varphi Y\\ \partial_r X & \partial_r Y \end{vmatrix} \equiv \frac{r}{\omega(r)}.
\end{gather*}
Thus, system \eqref{Ex0} in the variables $(r,\varphi)$ takes the form \eqref{PS} with $q=2$, $s_0=3/2$, $s_i=0$,
\begin{gather}\label{fgex0}
f(r,\varphi,S,t)\equiv t^{-\frac{1}{2}}f_1(r,\varphi,S), \quad 
g(r,\varphi,S,t)\equiv t^{-\frac{1}{2}}g_1(r,\varphi,S),
\end{gather}
where
\begin{align*}
 f_1(r,\varphi,S)&\equiv  r^{-1}Y(\varphi,r)Z(X(\varphi,r),Y(\varphi,r),S), \\ 
 g_1(r,\varphi,S)&\equiv - r^{-1}\omega(r)\partial_r X(\varphi,r)Z(X(\varphi,r),Y(\varphi,r),S).
\end{align*}
Note that $0<\omega(r)<1$ for all $0<|r|<(2\vartheta)^{-1/2}$. Hence, there exist $\kappa$, $\varkappa\in\mathbb Z_+$ and $0<|r|<(2\vartheta)^{-1/2}$ such that the condition \eqref{rc} holds. If $Z(x,y,S)\equiv0$, then $r(t)\equiv r_0$ and $\varphi(t)\equiv \omega(r_0)t+\phi_0$ with arbitrary constants $r_0$ and $\phi_0$. In the absence of the oscillatory part of the perturbation ($\alpha_1=\beta_1=0$), the amplitude of the solutions may tend to zero or to infinity, depending on the sign of $\beta_0$ (see~Fig.~\ref{FigEx0}, a). Under some conditions on the parameters, this qualitative behaviour can be preserved in the system with the oscillating perturbations (see~Fig.~\ref{FigEx0}, b), or violated with the appearance of new attracting states (see~Fig.~\ref{FigEx0}, c). The paper discusses the conditions that guarantee the existence and stability of such states in perturbed systems of the form \eqref{PS} with perturbations satisfying \eqref{fgas} and \eqref{Sform}.
\begin{figure}
\centering
\subfigure[$\alpha_1=0$]{
\includegraphics[width=0.3\linewidth]{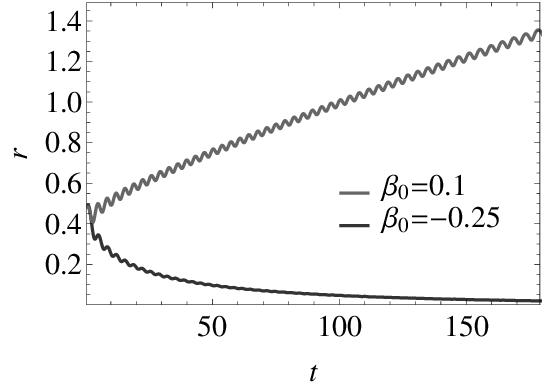}
}
\hspace{1ex}
 \subfigure[$\alpha_1=0.4$]{
 \includegraphics[width=0.3\linewidth]{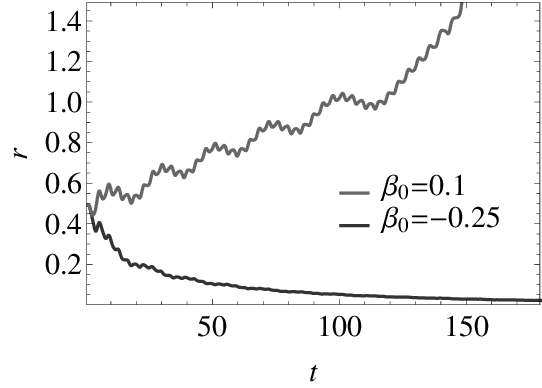}
}
\hspace{1ex}
\subfigure[$\alpha_1=0.6$]{
 \includegraphics[width=0.3\linewidth]{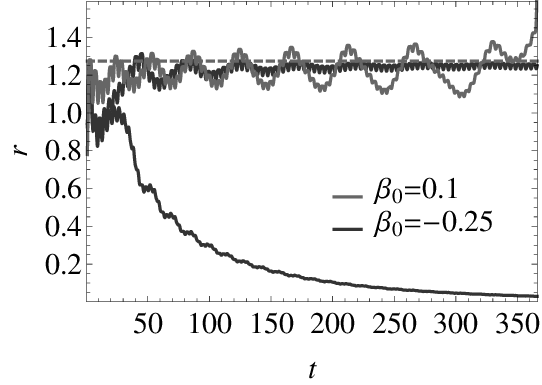}
}
\caption{\small The evolution of $r(t)=\sqrt{2U(x(t))+y^2(t)}$ for solutions of system \eqref{Ex0} with $\vartheta=1/4$, $\alpha_0=0.5$, $\beta_1=0$ and different values of the parameters $\alpha_1$, $\beta_0$. The dashed curve corresponds to $r(t)\equiv 1.27$.} \label{FigEx0}
\end{figure}

\section{Main results}\label{sec2}

Define the domain
\begin{align*}
&\mathcal D^\epsilon_{\varsigma,\tau}:=\{(R,\Psi,t)\in\mathbb R^3: |R t^{-\frac{1-\varsigma}{2q}}+a|< \mathcal R - \epsilon t^{-\frac{1-\varsigma}{2q}}, \ \  t\geq \tau \}
\end{align*} 
with some $\varsigma\in (0,1)$, $\epsilon\geq 0$ and $\tau>0$.  Let the angle brackets denote averaging of any function $F(S)$ over $S$ for the period $2\pi\varkappa$,
\begin{gather*}
\langle F(S)\rangle_{\varkappa S}\equiv \frac{1}{2\pi\varkappa}\int\limits_0^{2\pi\varkappa} F(S)\,dS.
\end{gather*}
Then, we have the following:
\begin{Th}\label{Th1}
Let system \eqref{PS} satisfy \eqref{fgas}, \eqref{Sform} and \eqref{rc}. Then for all $N\in\mathbb Z_+$, $\varsigma\in (0,1)$ and $\epsilon\in(0,\mathcal R)$ there exist $t_0\geq 1$ and the transformations $(r,\varphi)\mapsto (R,\Psi)\mapsto (\rho,\psi)$,
\begin{gather}
\label{ch1} r(t)=a+t^{-\frac{1}{2q}} R(t), \quad \varphi(t)=\frac{\kappa}{\varkappa}S(t)+\Psi(t), \\
\label{ch2} \rho(t)=R(t)+\tilde \rho_N(R(t),\Psi(t),t), \quad \psi(t)=\Psi(t)+\tilde \psi_N(R(t),\Psi(t),t),
\end{gather}
where $\tilde \rho_N(R,\Psi,t)$, $\tilde \psi_N(R,\Psi,t)$ are $2\pi$-periodic in $\Psi$ and satisfy the inequalities
\begin{gather*}
|\tilde \rho_N(R,\Psi,t)|\leq \epsilon, \quad |\tilde \psi_N(R,\Psi,t)|\leq \epsilon, \quad (R,\Psi,t)\in\mathcal D^0_{\varsigma,t_0},
\end{gather*}
such that for all $0<|r|< \mathcal R$ and $\varphi\in\mathbb R$ system \eqref{PS} can be transformed into
\begin{gather}\label{rhopsi}
\frac{d\rho}{dt}=\Lambda_{N}(\rho,\psi,S(t),t), \quad 
\frac{d\psi}{dt}=\Omega_N(\rho,\psi,S(t),t),
\end{gather}
with $\Lambda_{N}(\rho,\psi,S,t)\equiv \hat\Lambda_N(\rho,\psi,t)+ \tilde\Lambda_N(\rho,\psi,S,t)$, $\Omega_{N}(\rho,\psi,S,t)\equiv \hat\Omega_N(\rho,\psi,t)+\tilde\Omega_N(\rho,\psi,S,t)$, defined for all $(\rho,\psi,t)\in\mathcal D^\epsilon_{\varsigma,t_0}$ and $S\in\mathbb R$, such that
\begin{gather}
\label{tildeLON}
\hat\Lambda_N(\rho,\psi,t)\equiv \sum_{k=1}^N t^{-\frac{k}{2q}}\Lambda_k(\rho,\psi), \quad 
\hat\Omega_N(\rho,\psi,t)\equiv \sum_{k=1}^N t^{-\frac{k}{2q}}\Omega_k(\rho,\psi),\\
\label{tildeLO}\tilde\Lambda_N(\rho,\psi,S,t)=\mathcal O(t^{-\frac{N+1}{2q}}), \quad
\tilde\Omega_N(\rho,\psi,S,t)=\mathcal O(t^{-\frac{N+1}{2q}})
\end{gather}
as $t\to\infty$ uniformly for all $|\rho|<\infty$ and $(\psi,S)\in\mathbb R^2$, where $\Lambda_k(\rho,\psi)$ and  $\Omega_k(\rho,\psi)$ are polynomials in $\rho$ of degree $k-1$ and $k$, respectively. 
In particular, $\Lambda_1(\rho,\psi)\equiv \langle f_1(a,\kappa S/\varkappa+\psi,S)\rangle_{\varkappa S}$ and $\Omega_1(\rho,\psi)\equiv \eta\rho$. Moreover,   
\begin{gather}\label{LOas}
\begin{split}
	&\tilde \Lambda_N(\rho,\psi,S,t)\sim  \sum_{k=N+1}^\infty t^{-\frac{k}{2q}}\tilde \Lambda_{N,k}(\rho,\psi,S), \\ 
	&\tilde \Omega_N(\rho,\psi,S,t) \sim \sum_{k=N+1}^\infty t^{-\frac{k}{2q}} \tilde \Omega_{N,k}(\rho,\psi,S), \quad t\to\infty,
	\end{split}
\end{gather}
where $\tilde \Lambda_{N,k}(\rho,\psi,S)=\mathcal O(\rho^{k-1})$ and $\tilde \Omega_{N,k}(\rho,\psi,S)=\mathcal O(\rho^{k})$ as $\rho\to\infty$ uniformly for all $(\psi,S)\in\mathbb R^2$. 
\end{Th}
The proof is contained in Section~\ref{sec3}. 

Note that Theorem~\ref{Th1} describes an averaging transformation that simplifies the system in the leading asymptotic terms as $t\to\infty$. Moreover, after this procedure, some terms in sums \eqref{tildeLON} may disappear because they have the zero mean. Let $n\in [1,2q]$ and $m\in[2,2q]$ be integers such that
\begin{gather}\label{asn}
\begin{split}
&\Lambda_i(\rho,\psi)\equiv 0, \quad  1\leq i<n,  \quad \Lambda_n(\rho,\psi)\not\equiv 0,\\
&\Omega_j(\rho,\psi)\equiv 0, \quad 1<j< m,  \quad \Omega_m(\rho,\psi)\not\equiv 0.
\end{split}
\end{gather}

The proposed method is based on the study of the truncated system
\begin{gather}\label{trsys}
\frac{d\varrho}{dt}=\hat\Lambda_N(\varrho,\phi,t), \quad \frac{d\phi}{dt}=\hat\Omega_N(\varrho,\phi,t)
\end{gather}
obtained from \eqref{rhopsi} by dropping the remainder terms $\tilde \Lambda_N$ and $\tilde \Omega_N$. System \eqref{trsys} can be considered as a model system that describes the average dynamics for the amplitude residual and phase shift. First, we discuss the solutions of system \eqref{trsys}. Next, we show that the trajectories of the full system \eqref{rhopsi} behave like the solutions of the truncated system.

The behaviour of solutions to asymptotically autonomous system \eqref{trsys} depends on the properties of the corresponding limiting system 
\begin{gather}\label{limsys}
t^{\frac{n}{2q}} \frac{d\hat\varrho}{dt}=\Lambda_n(\hat\varrho,\hat\phi), \quad t^{\frac{1}{2q}} \frac{d\hat\phi}{dt}=\eta\hat\varrho.
\end{gather}
In particular, the presence and the stability of fixed points in system \eqref{limsys} play a crucial role.  With this in mind, we consider the following assumption:
\begin{gather}
\label{aszero}
\exists \, \psi_0\in\mathbb R: \quad \Lambda_n(0,\psi_0)=0, \quad \nu_n:=\partial_\psi \Lambda_n(0,\psi_0)\neq 0,
\end{gather}
and define the parameter $\lambda_n:=\partial_\rho \Lambda_n(0,\psi_0)$. In this case, system \eqref{limsys} has an equilibrium $(0,\psi_0)$, and we have
\begin{Lem}\label{Lem01}
Let assumptions \eqref{asn} and \eqref{aszero} hold.
 \begin{itemize}
	\item If $\nu_n \eta>0$ or $ \nu_n \eta<0$, $\lambda_n>0$, $n\leq q$, then the equilibrium $(0,\psi_0)$ of system \eqref{limsys} is unstable.
	\item If $ \nu_n \eta<0$ and $\lambda_n<0$, then the equilibrium $(0,\psi_0)$ of system \eqref{limsys} is asymptotically stable.
\end{itemize}
 \end{Lem}
Note that if $ \nu_n \eta>0$, the equilibrium $(0,\psi_0)$ is of saddle type. In this case, similar dynamics occurs in the full system. However, if $ \nu_n \eta<0$, the fixed point can be either stable or unstable, depending on the sign of the divergence of the vector field calculating at the equilibrium. Let us show that under a similar condition there exists a solution of system \eqref{trsys} tending to the point $(0,\psi_0)$ as $t\to\infty$. Define 
\begin{gather*}
d_{n,m}:=\begin{cases} 
\lambda_n, & n<m,\\
\lambda_n+\omega_m, & n=m,\\
\omega_m, & n>m,
\end{cases}
\end{gather*}
where $\omega_m:=\partial_\varphi \Omega_m(0,\psi_0)$. Then, we have the following:
\begin{Lem}\label{Lem1}
Let assumptions \eqref{asn} and \eqref{aszero} hold with $\nu_n\eta<0$ and $d_{n,m}<0$. Then for all $N\geq \max\{m,n\}$ system \eqref{trsys} has a particular solution $\varrho_\ast(t)$, $\phi_\ast(t)$ with asymptotic expansion in the form
\begin{gather}\label{assol}
\varrho_\ast(t)\sim \sum_{k=1}^\infty t^{-\frac{k+m-2}{2q}} \varrho_k, \quad \phi_\ast(t)\sim \psi_0+\sum_{k=1}^\infty t^{-\frac{k}{2q}}\phi_k, \quad t\to\infty,
\end{gather}
where $\varrho_k$, $\phi_k$ are some constants. Moreover, the solution $\varrho_\ast(t)$, $\phi_\ast(t)$ is asymptotically stable.
\end{Lem}

It can be shown that the dynamics described by the solution $\varrho_\ast(t)$, $\phi_\ast(t)$ of the truncated system is preserved in system \eqref{rhopsi}. We have
\begin{Th}\label{Th2}
Let system \eqref{PS} satisfy \eqref{fgas}, \eqref{Sform}, \eqref{rc}, and assumptions \eqref{asn} and \eqref{aszero} hold with $\nu_n\eta<0$ and $d_{n,m}<0$. Then, there is
$N_0\in\mathbb Z_+$ such that for all $N\geq N_0$ and $\varepsilon>0$ there exist $\delta>0$ and $t_\ast\geq t_0$ such that any solution $r(t)$, $\varphi(t)$ of system \eqref{rhopsi} with initial data $r(t_\ast)=r_\ast$, $\varphi(t_\ast)=\varphi_\ast$, 
$ |r_\ast-a-t_\ast^{-1/(2q)}\varrho_\ast(t_\ast) |+ |\varphi_\ast-\kappa S(t_\ast)/\varkappa-\phi_\ast(t_\ast)|\leq \delta $, satisfies the inequality: 
\begin{gather}\label{rpst}
\left|r(t)-a-t^{-\frac{1}{2q}}\varrho_\ast(t)\right|+\left|\varphi(t)-\frac{\kappa}{\varkappa}S(t)-\phi_\ast(t)\right|<\varepsilon
\end{gather} for all $t>t_\ast$.
\end{Th}

Note that in the opposite case, when $\nu_n\eta<0$ and $d_{n,m}>0$, the asymptotic regime described in Lemma~\ref{Lem1} turns out to be unstable. Let $\varrho_{\ast,M}(t)$, $\phi_{\ast,M}(t)$ be $M$th partial sums of the series \eqref{assol} and $\ell=\min\{m,n\}$. Then, we have
\begin{Th}\label{Th21}
Let system \eqref{PS} satisfy \eqref{fgas}, \eqref{Sform}, \eqref{rc}, and assumptions \eqref{asn} and \eqref{aszero} hold with  $\nu_n\eta<0$, $d_{n,m}>0$, and $\ell+n-1< 2 q$. 
Then, there is $N_0\in\mathbb Z_+$, $\varepsilon > 0$ such that for all $\delta\in (0,\varepsilon)$, $N\geq N_0$, $M\in \mathbb Z_+$ the solution $r(t)$, $\varphi(t)$ of system \eqref{rhopsi} with initial data $r(t_\ast)=r_\ast$, $\varphi(t_\ast)=\varphi_\ast$,
$ |r_\ast-a-t_\ast^{-1/(2q)}\varrho_\ast(t_\ast) |+ |\varphi_\ast-\kappa S(t_\ast)/\varkappa-\phi_\ast(t_\ast)|\leq \delta $
at some $t_\ast\geq t_0$, satisfies the inequality: 
\begin{gather}\label{rpunst}
\left|r(t_e)-a-t_e^{-\frac{1}{2q}}\varrho_\ast(t_e)\right|+\left|\varphi(t_e)-\frac{\kappa}{\varkappa}S(t_e)-\phi_\ast(t_e)\right|\geq \varepsilon
\end{gather}
at some $t_e>t_\ast$.
\end{Th}

Let us remark that if $d_{n,m}=0$, the existence and stability of the phase locking regime are not guaranteed by Theorem~\ref{Th2}. For this case, consider the following assumption:
\begin{gather}\label{asst}
\begin{split}
 &\exists\,  h\in (\ell,2q]: \quad \partial_\rho \Lambda_k(\rho,\psi)+\partial_\psi \Omega_k(\rho,\psi)\equiv 0, \quad k\leq h-1, \\
 & d_h:= \partial_\rho \Lambda_h(0,\psi_0)+\partial_\psi \Omega_h(0,\psi_0)\neq 0.
\end{split}
\end{gather}
Then, we have
\begin{Lem}\label{Lem2}
Let assumptions \eqref{asn}, \eqref{aszero} and \eqref{asst} hold with $\nu_n\eta<0$ and $d_h<0$. Then for all $N\geq \max\{m,n,h\}$ system \eqref{trsys} has a stable particular solution $\varrho_\ast(t)$, $\phi_\ast(t)$ with asymptotic expansion \eqref{assol}. Moreover, the solution $\varrho_\ast(t)$, $\phi_\ast(t)$ is asymptotically stable if $h+n-1<2q$.
\end{Lem}

As in the previous case, the phase locking regime in system \eqref{PS} associated with the solution $\varrho_\ast(t)$, $\phi_\ast(t)$ of the model system \eqref{trsys} turns out to be stable if $d_h<0$ and unstable if $d_h>0$. We have the following:
\begin{Th}\label{Th23}
Let system \eqref{PS} satisfy \eqref{fgas}, \eqref{Sform}, \eqref{rc}, and assumptions \eqref{asn}, \eqref{aszero} and \eqref{asst} hold with $\nu_n\eta<0$ and $d_{h}<0$. Then there is
$N_0\in\mathbb Z_+$ such that for all $N\geq N_0$ and $\varepsilon>0$ there exist $\delta>0$ and $t_\ast\geq t_0$ such that any solution $r(t)$, $\varphi(t)$ of system \eqref{rhopsi} with initial data $r(t_\ast)=r_\ast$, $\varphi(t_\ast)=\varphi_\ast$, $|r_\ast-a-t_\ast^{-1/(2q)}\varrho_\ast(t_\ast) |+ |\varphi_\ast-\kappa S(t_\ast)/\varkappa-\phi_\ast(t_\ast)|\leq \delta $, satisfies the inequality \eqref{rpst} for all $t>t_\ast$.
\end{Th}

\begin{Th}\label{Th24}
Let system \eqref{PS} satisfy \eqref{fgas}, \eqref{Sform}, \eqref{rc}, and assumptions \eqref{asn}, \eqref{aszero} and \eqref{asst} hold with  $\nu_n\eta<0$, $d_{h}>0$, and $h+n-1< 2 q$. 
Then, there is $N_0\in\mathbb Z_+$, $\varepsilon > 0$ such that for all $\delta\in (0,\varepsilon)$, $N\geq N_0$, $M\in \mathbb Z_+$ the solution $r(t)$, $\varphi(t)$ of system \eqref{rhopsi} with initial data $r(t_\ast)=r_\ast$, $\varphi(t_\ast)=\varphi_\ast$,
$ |r_\ast-a-t_\ast^{-1/(2q)}\varrho_\ast(t_\ast) |+ |\varphi_\ast-\kappa S(t_\ast)/\varkappa-\phi_\ast(t_\ast)|\leq \delta $
at some $t_\ast\geq t_0$, satisfies the inequality \eqref{rpunst} at some $t_e>t_\ast$.
\end{Th}

Thus, under the assumptions of Theorems~\ref{Th2} and~\ref{Th23}, it follows that there exists a stable phase locking regime in system \eqref{PS} with $r(t)\approx a$ and $\varphi(t)\approx \kappa S(t)/\varkappa+\psi_0$ as $t\to\infty$. 

Consider finally the case when, instead of \eqref{aszero}, the following assumption holds:
\begin{gather}\label{asnzero} 
\Lambda_n(\rho,\psi)\neq 0 \quad \forall\, (\rho,\psi)\in\mathbb R^2.
\end{gather}
Then we have 
\begin{Th}\label{Th3}
Let system \eqref{PS} satisfy \eqref{fgas}, \eqref{Sform}, \eqref{rc}, and assumptions \eqref{asn} and \eqref{asnzero} hold. Then the solutions of system \eqref{rhopsi} exit from any bounded domain in a finite time. 
\end{Th}
In this case, $\varphi(t)$ for solutions of system \eqref{PS} can significantly differ from the phase $\kappa S(t)/\varkappa$, and the solutions with $r(t)\approx a$ does not occur.  

\section{Change of variables}
\label{sec3}

\subsection{Amplitude residual and phase shift}
Substituting \eqref{ch1} into \eqref{PS} yields the following system:
\begin{gather}\label{RPsys}
\frac{dR}{dt}=F(R,\Psi,S(t),t), \quad \frac{d\Psi}{dt}=G(R,\Psi,S(t),t),
\end{gather}
where
\begin{gather}\label{FG}
\begin{split}
&F(R,\Psi,S,t)\equiv t^{\frac{1}{2q}} f\left(a+t^{-\frac{1}{2q}} R,\frac{\kappa}{\varkappa}S +\Psi,S,t\right)+t^{-1}\frac{R}{2q}, \\
& G(R,\Psi,S,t)\equiv \omega\left(a+t^{-\frac{1}{2q}} R\right)-\frac{\kappa}{\varkappa}S'(t)+g\left(a+t^{-\frac{1}{2q}} R,\frac{\kappa}{\varkappa}S +\Psi,S,t\right).
\end{split}
\end{gather}
It follows from \eqref{fgas} and \eqref{Sform} that the functions $F(R,\Psi,S,t)$ and $G(R,\Psi,S,t)$ have the following asymptotic expansion:
\begin{gather}\label{FGas}\begin{split}
F(R,\Psi,S,t)&\sim \sum_{k=1}^\infty t^{-\frac{k}{2q}}F_k(R,\Psi,S), \\ 
G(R,\Psi,S,t)&\sim \sum_{k=1}^\infty t^{-\frac{k}{2q}}G_k(R,\Psi,S), \quad 
t\to\infty,
\end{split}
\end{gather}
where the coefficients
\begin{align*}
& F_k(R,\Psi,S)\equiv \sum_{\substack{ i+2j=k+1 \\ i\geq 0, j\geq 1}} \partial_r^i f_j\left(a,\frac{\kappa}{\varkappa}S +\Psi,S\right) \frac{R^i}{i!} + \delta_{k,2q}\frac{ R}{2q},\\
& G_k(R,\Psi,S)\equiv \partial_r^k\omega(a)\frac{R^k}{k!}-\frac{\kappa}{\varkappa} s_{k/2}\left(1-\frac{k}{2q}+\delta_{k,2q}\right)+\sum_{\substack{ i+2j=k \\ i\geq 0, j\geq 1}} \partial_r^i g_j\left(a,\frac{\kappa}{\varkappa}S +\Psi,S\right) \frac{R^i}{i!}
\end{align*}
are $2\pi$-periodic in $\Psi$ and $2\pi \varkappa$-periodic in $S$. Here $\delta_{k,2q}$ is the Kronecker delta. We set $s_j=0$ for $j>q$ and $s_{k/2}=0$ for odd $k$. 
Note that 
$F_1(R,\Psi,S)\equiv f_1(a,\kappa S/\varkappa+\Psi,S)$ and 
$G_1(R,\Psi,S)\equiv \eta R$.
Since $F_k(R,\Psi,S)=\mathcal O(R^{k-1})$ and $G_k(R,\Psi,S)=\mathcal O(R^{k})$ as $R\to\infty$ uniformly for all $(\Psi,S)\in\mathbb R^2$, we see that the asymptotic approximations \eqref{FGas} for the right-hand side of system \eqref{FG} are applicable for all $(R,\Psi,t)\in\mathcal D^0_{\varsigma,\tau}$ with $\varsigma \in (0,1)$ and some $\tau\geq 1$. 

\subsection{Near identity transformation}
We see that system \eqref{RPsys} is asymptotically autonomous with the limiting system 
\begin{gather*}
\frac{d\hat R}{dt}=0, \quad \frac{d\hat\Psi}{dt}=0, \quad \frac{d\hat S}{dt}=s_0.
\end{gather*}
Hence, the phase $S(t)$ can be considered as an analogue of a fast variable as $t\to\infty$ in comparison with the solutions $R(t)$, $\Psi(t)$ of system \eqref{RPsys}. This can be used to simplify the system by averaging the equations with respect to the variable $S(t)$. Note that such method is effective in similar problems with a small parameter (see, for instance,~\cite{BM61,AKN06}). The transformation is sought in the following form:
\begin{gather}\label{rpch}\begin{split}
	U_N(R,\Psi,S,t)&=R+\sum_{k=1}^N t^{-\frac{k}{2q}} u_k(R,\Psi,S), \\
	V_N(R,\Psi,S,t)&=\Psi+\sum_{k=1}^N t^{-\frac{k}{2q}} v_k(R,\Psi,S)
\end{split}
\end{gather}
with some integer $N\geq 1$. The coefficients $u_k(R,\Psi,S)$, $v_k(R,\Psi,S)$ are assumed to be periodic with respect to $\Psi$ and $S$, and are chosen in such a way that the system in the new variables 
\begin{gather*}
\rho(t)\equiv U_N(R(t),\Psi(t),S(t),t), \quad  
\psi(t)\equiv V_N(R(t),\Psi(t),S(t),t)
\end{gather*}
takes the form \eqref{rhopsi}, where the right-hand sides do not depend explicitly on $S(t)$ at least in the first $N$ terms of the asymptotics as $t\to\infty$. Differentiating \eqref{rpch} with respect to $t$ and taking into account \eqref{Sform}, \eqref{RPsys} and \eqref{FGas}, we get
\begin{gather}\label{drpas}
\begin{split}
\frac{d}{dt}\begin{pmatrix} U_N  \\ V_N \end{pmatrix} \equiv &\left(\frac{dR}{dt}\partial_R+\frac{d\Psi}{dt}\partial_\Psi+\frac{dS}{dt}\partial_S+\partial_t\right)\begin{pmatrix} U_N  \\ V_N \end{pmatrix}\\
\sim &
\sum_{k=1}^\infty t^{-\frac{k}{2q}} \left\{\begin{pmatrix} F_k \\ G_k \end{pmatrix}+s_0\partial_S \begin{pmatrix} u_k \\ v_k \end{pmatrix}\right\}\\& + 
\sum_{k=1}^\infty t^{-\frac{k}{2q}}\sum_{j=1}^{k-1}\left\{ F_j\partial_R + G_j\partial_\Psi+s_{j/2}\left(1-\frac{j}{2q}+\delta_{j,2q}\right)\partial_S+\delta_{j,2q} \frac{2q-k}{2q}\right\}\begin{pmatrix} u_{k-j} \\ v_{k-j} \end{pmatrix}
\end{split}
\end{gather}
as $t\to\infty$, where it is assumed that $u_k(R,\Psi,S)\equiv v_k(R,\Psi,S)\equiv 0$ for $k\leq 0$ and $k>N$. Comparing the coefficients of powers of $t^{-1/2q}$ in \eqref{rhopsi} and \eqref{drpas} yields
\begin{gather}\label{ukvk}
s_0\partial_S \begin{pmatrix} u_k \\ v_k \end{pmatrix}=\begin{pmatrix} \Lambda_k(R,\Psi)-F_k(R,\Psi,S) +\tilde F_k(R,\Psi,S)\\ \Omega_k(R,\Psi)-G_k(R,\Psi,S)+\tilde G_k(R,\Psi,S) \end{pmatrix}, \quad k=1,\dots, N,
\end{gather} 
where the functions $\tilde F_k(R,\Psi,S)$, $\tilde G_k(R,\Psi,S)$ are expressed in terms of $\{u_{j}, v_{j},\Lambda_{j}, \Omega_{j}\}_{j=1}^{k-1}$ by the following formulas:
\begin{align*}
\begin{pmatrix}
\tilde F_1 \\ \tilde G_1
\end{pmatrix} \equiv 
& \begin{pmatrix}
0\\ 0
\end{pmatrix},\\
\begin{pmatrix}
\tilde F_2 \\ \tilde G_2
\end{pmatrix} \equiv 
& 
(u_1\partial_R+v_1\partial_\Psi)
\begin{pmatrix}
\Lambda_1\\ \Omega_1
\end{pmatrix}
-(F_1\partial_R + G_1\partial_\Psi )\begin{pmatrix} u_{1} \\ v_{1} \end{pmatrix}, \\
\begin{pmatrix}
\tilde F_3 \\ \tilde G_3
\end{pmatrix} \equiv 
& 
\sum_{i+j=3}(u_i\partial_R+v_i\partial_\Psi)
\begin{pmatrix}
\Lambda_j\\ \Omega_j
\end{pmatrix} + \frac{1}{2}\left(u_1^2\partial_R^2+2u_1v_1 \partial_R\partial_\Psi+v_1^2 \partial_\Psi^2\right)\begin{pmatrix}
\Lambda_1\\ \Omega_1
\end{pmatrix} \\
& -\sum_{j=1}^2 \left\{ F_j\partial_R + G_j\partial_\Psi+s_{j/2}\left(1-\frac{j}{2q}+\delta_{j,2q}\right)\partial_S+\delta_{j,2q} \frac{2q-3}{2q}\right\}\begin{pmatrix} u_{3-j} \\ v_{3-j} \end{pmatrix},\\
\begin{pmatrix}
\tilde F_k \\ \tilde G_k
\end{pmatrix} \equiv 
& 
\sum_{  m_1+\cdots im_i+n_1+\cdots l n_l+j=k } C_{i,l,m_1,\dots,m_i,n_1,\dots,n_l} u_1^{m_1} \cdots u_i^{m_i} v_1^{n_1} \cdots v_l^{n_l} \partial_R^{m_1+\cdots+m_i}\partial_\Psi^{n_1+\cdots+n_l}
\begin{pmatrix}
\Lambda_j\\ \Omega_j
\end{pmatrix}  \\
& -\sum_{j=1}^{k-1} \left\{ F_j\partial_R + G_j\partial_\Psi+s_{j/2}\left(1-\frac{j}{2q}+\delta_{j,2q}\right)\partial_S+\delta_{j,2q} \frac{2q-k}{2q}\right\}\begin{pmatrix} u_{k-j} \\ v_{k-j} \end{pmatrix}
\end{align*}
with some constant parameters $C_{i,l,m_1,\dots,m_i,n_1,\dots,n_l}$. To avoid the appearance of secular terms in \eqref{rpch} and guarantee the existence of periodic solutions to system \eqref{ukvk}, we take
\begin{align*}
\Lambda_k(R,\Psi)&\equiv \langle F_k(R,\Psi,S)-\tilde F_k(R,\Psi,S)\rangle_{\varkappa S},\\
\Omega_k(R,\Psi)&\equiv \langle G_k(R,\Psi,S)-\tilde G_k(R,\Psi,S)\rangle_{\varkappa S}.
\end{align*}
In particular, $\Lambda_1(R,\Psi)\equiv \langle f_1(a,\kappa S/\varkappa+\Psi,S)\rangle_{\varkappa S}$ and $\Omega_1(R,\Psi)\equiv \eta R$.
Hence, system \eqref{ukvk} is solvable in the class of functions that are $2\pi\varkappa$-periodic in $S$ with zero mean. Moreover, it is not hard to check that $u_k(R,\Psi,S)$, $v_k(R,\Psi,S)$, $\Lambda_k(R,\Psi)$, $\Omega_k(R,\Psi)$ are $2\pi$-periodic in $\Psi$ and 
\begin{align*}
&\tilde F_k(R,\Psi,S)= \mathcal O(R^{k-1}),& \quad & \Lambda_k(R,\Psi)= \mathcal O(R^{k-1}), &\quad & u_k(R,\Psi,S)= \mathcal O(R^{k-1}), \\
&\tilde G_k(R,\Psi,S)= \mathcal O(R^{k}), & \quad & \Omega_k(R,\Psi)= \mathcal O(R^{k}), &\quad & v_k(R,\Psi,S)= \mathcal O(R^{k})
\end{align*}
as $R\to\infty$ uniformly for all $(\Psi,S)\in\mathbb R^2$.  This together with \eqref{rpch} implies that for all $\epsilon\in (0,\mathcal R)$ there exists $t_0\geq 1$ such that
\begin{align*}
&|U_N(R,\Psi,S,t)-R|\leq \epsilon, &\quad &|\partial_R U_N(R,\Psi,S,t)-1|\leq \epsilon, &\quad &|\partial_\Psi U_N(R,\Psi,S,t)|\leq \epsilon,\\
&|V_N(R,\Psi,S,t)-\Psi|\leq \epsilon, &\quad & |\partial_R V_N(R,\Psi,S,t)|\leq \epsilon, &\quad & |\partial_\Psi V_N(R,\Psi,S,t)-1|\leq \epsilon 
\end{align*}
for all $(R,\Psi,t)\in\mathcal D^0_{\varsigma,t_0}$, $S\in\mathbb R$ and $\varsigma\in (0,1)$. Thus, \eqref{ch2} is invertible.  Denote by $R= u(\rho, \psi, t)$, $\Psi=v(\rho,\psi,t)$ the corresponding inverse
transformation defined for all $(\rho,\psi,t)\in\mathcal D^\epsilon_{\varsigma,t_0}$.
Then, 
\begin{align*}
\begin{pmatrix}
\tilde \Lambda_N(\rho,\psi,S,t)\\
\tilde \Omega_N(\rho,\psi,S,t)
\end{pmatrix}
\equiv &
\left(\partial_t+F \partial_R+G\partial_\Psi\right)
\begin{pmatrix}
U_N\\
V_N
\end{pmatrix}\Big|_{R=u(\rho, \psi, t), \Psi=v(\rho,\psi,t)}  - \sum_{k=1}^N t^{-\frac{k}{2q}} \begin{pmatrix}
\Lambda_k(\rho,\psi)\\
\Omega_k(\rho,\psi).
\end{pmatrix}
\end{align*}
Combining this with \eqref{drpas}, we get \eqref{LOas}.

Thus, we obtain the proof of Theorem~\ref{Th1} with $\tilde \rho_N(R,\Psi,t)\equiv  U_N(R,\Psi,S(t),t)-R$, $\tilde \psi_N(R,\Psi,t)\equiv  V_N(R,\Psi,S(t),t)-\Psi$.

\section{Analysis of the model system}\label{sec4}

\begin{proof}[Proof of Lemma~\ref{Lem01}]
Substituting $\hat \varrho(t)=u(t)$, $\hat \phi(t)=\psi_0+v(t) $ into \eqref{limsys} yields the following system with an equilibrium at $(0, 0)$:
\begin{gather}\label{limsys0}
\frac{du}{dt}=t^{-\frac{n}{2q}}\Lambda_n(u,\psi_0+v), \quad \frac{dv}{dt}=t^{-\frac{1}{2q}}\eta u.
\end{gather}
Consider the linearised system 
\begin{gather*}
\frac{d {\bf z}}{dt}= {\bf M}(t)
{\bf z}, \quad 
{\bf M}(t)\equiv \begin{pmatrix} t^{-\frac{n}{2q}}\lambda_n & t^{-\frac{n}{2q}} \nu_n \\
t^{-\frac{1}{2q}} \eta & 0
\end{pmatrix},
\quad
{\bf z}=\begin{pmatrix}u\\v\end{pmatrix}.
\end{gather*}
The roots of the characteristic equation $|{\bf M}(t)-\mu {\bf I}|=0$ are given by
\begin{gather*}
\mu_\pm(t)=\frac{t^{-\frac{n}{2q}}}{2}  \left(\lambda_n \pm \sqrt{4\nu_n\eta t^{\frac{n-1}{2q}}+\lambda_n^2}\right).
\end{gather*}
We see that if $\nu_n \eta>0$, the eigenvalues $\mu_+(t)$ and $\mu_-(t)$ are real of different signs. This implies that the equilibrium is of saddle type and the fixed point $(0,\psi_0)$ of system \eqref{limsys} is unstable. 

Let us show that in the opposite case, when $\nu_n \eta<0$, the stability of the equilibrium depends on the sign of $\lambda_n\neq 0$. Consider first the case $n=1$. We use 
\begin{gather}\label{LF1}
L_1(u,v)\equiv \frac{1}{2}\left(|\eta| u^2+|\nu_1|v^2\right)+ \chi_1  u v
\end{gather}
as a Lyapunov function candidate for system \eqref{limsys0}, where $\chi_1\in\mathbb R$ is a parameter such that
\begin{gather}\label{chi1}
	{\hbox{\rm sgn} }\, \chi_1 = {\hbox{\rm sgn} }\, (\nu_1  \lambda_1), \quad 
	|\chi_1|=\frac 12\min\left\{|\eta|,|\nu_1|,\frac{2|\lambda_1 \eta \nu_1|}{\lambda_1^2+2|\eta\nu_1|}\right\}.
\end{gather}
It can easily be checked that there exists $\Delta_0>0$ such that 
\begin{gather}\label{L1ineq}
	L_- \Delta^2\leq L_1(u,v)\leq L_+ \Delta^2
\end{gather}
for all $(u,v)\in\mathbb R^2$ such that $\Delta=\sqrt{u^2+v^2}\leq \Delta_0$, where $L_-=\min\{|\eta|-|\chi_1|,|\nu_1|-|\chi_1|\}/4>0$ and $L_+=\max\{|\eta|+|\chi_1|,|\nu_1|+|\chi_1|\}$. The derivative of $L_1(u,v)$ with respect to $t$ along the trajectories of the system satisfies
\begin{gather*}
\frac{dL_1}{dt}\Big|_\eqref{limsys0} = t^{-\frac{1}{2q}}\left( (\lambda_1-({\hbox{\rm sgn} }\, \lambda_1)|\chi_1|) |\eta|u^2+({\hbox{\rm sgn} }\, \lambda_1)|\chi_1 \nu_1|v^2+\chi_1 \lambda_1  uv +\mathcal O(\Delta^3)\right), \quad \Delta\to 0.
\end{gather*}
Using Young’s inequality, we obtain
\begin{align*}
\frac{dL_1}{dt}\Big|_\eqref{limsys0} &\geq  t^{-\frac{1}{2q}} \left\{A_1 u^2+B_1 v^2+\mathcal O(\Delta^3)\right\}\quad  \text{if} \quad\lambda_1>0,\\
\frac{dL_1}{dt}\Big|_\eqref{limsys0} & \leq  -t^{-\frac{1}{2q}} \left\{ A_1 u^2+B_1 v^2+\mathcal O(\Delta^3)\right\}\quad   \text{if} \quad \lambda_1<0
\end{align*}
with positive parameters 
\begin{gather}\label{A1B1} 
A_1=\frac{|\chi_1| (\lambda_1^2+2|\nu_1 \eta|) }{2 |\nu_1|}, \quad B_1=\frac{|\chi_1 \nu_1|}{2}.
\end{gather}
Hence, there exists $0<\Delta_1\leq \Delta_0$ such that 
\begin{align} 
\label{dL1}&\frac{dL_1}{dt}\Big|_\eqref{limsys0}  \geq  \gamma_1 t^{-\frac{1}{2q}} L_1\geq 0 \quad \text{if} \quad \lambda_1>0, \\
\label{dL2}&\frac{dL_1}{dt}\Big|_\eqref{limsys0}   \leq  -\gamma_1 t^{-\frac{1}{2q}} L_1\leq 0 \quad \text{if} \quad \lambda_1<0
\end{align}
for all $(u,v)\in\mathbb R^2$ such that $\Delta\leq \Delta_1$ with $\gamma_1=\min\{A_1,B_1\}/(2 L_+)>0$. 
If $\lambda_1>0$, then integrating \eqref{dL1} with respect to $t$ and taking into account \eqref{L1ineq}, we obtain the instability of the equilibrium $(0,0)$ in system \eqref{limsys0} and the fixed point $(0,\psi_0)$ in system \eqref{limsys}. Indeed, there exists $\epsilon \in (0, \Delta_1)$ such that for all $\delta\in (0,\epsilon)$ the solution $(u(t),v(t))$ of \eqref{limsys0} with initial data $\sqrt{u^2(t_0)+v^2(t_0)}=\delta$ leaves the domain $\{(u,v)\in\mathbb R^2: \Delta\leq \epsilon\}$ as $t\geq t_1$, where
\begin{gather*}
t_1^{1-\frac{1}{2q}}=t_0^{1-\frac{1}{2q}}+\left(\frac{2q-1}{2q \gamma_1}\right)\log \left(\frac{L_+\epsilon^2}{L_-\delta^2}\right).
\end{gather*}
If $\lambda_1<0$, then it follows from \eqref{dL2} that for all $\epsilon\in (0,\Delta_1)$ there exists $\delta\in (0,\epsilon)$ such that the solution $(u(t),v(t))$ of \eqref{limsys0} with initial data $\sqrt{u^2(t_0)+v^2(t_0)}\leq \delta$ cannot exit from the domain $\{(u,v)\in\mathbb R^2: \Delta\leq \epsilon\}$. Hence, the equilibrium $(0,0)$ of system \eqref{limsys0} and the fixed point $(0,\psi_0)$ of system \eqref{limsys} are stable. Moreover, by integrating \eqref{dL2}, we obtain the inequality 
\begin{gather*}
L_1(u(t),v(t))\leq L_1(u(t_0),v(t_0)) \exp\left\{-\frac{2 q\gamma_1 }{2q-1}\left(t^{1-\frac{1}{2q}}-t_0^{1-\frac{1}{2q}}\right)\right\}, \quad t\geq t_0.
\end{gather*}
Combining this with \eqref{L1ineq}, we get asymptotic stability of the equilibrium.

Let $n\geq 2$. Consider
\begin{gather*}
L_n(u,v,t)\equiv t^{\frac{n-1}{2q}}\frac{|\eta| }{2} u^2+({\hbox{\rm sgn}}\, \nu_n) \int\limits_0^v \Lambda_n(u,\psi_0+w)\,dw  + t^{-\frac{n-1}{2q}}\left(\frac{\lambda_n^2 v^2}{2|\eta|}+\lambda_n ({\hbox{\rm sgn}}\, \eta) u v\right)
\end{gather*}
as a Lyapunov function candidate. 
Note that there exist $\Delta_0>0$ and $t_1\geq t_0$ such that 
\begin{gather*}
	\frac{1}{4}\left(|\nu_n|u^2+|\eta|v^2\right)\leq L_n(u,v,t)\leq t^{\frac{n-1}{2q}} \left(|\nu_n|u^2+|\eta|v^2\right)
\end{gather*}
for all $(u,v,t)\in\mathbb R^3$ such that $\Delta\leq \Delta_0$ and $t\geq t_1$. The derivative of $L_n(u,v,t)$ with respect to $t$ along the trajectories of system \eqref{limsys} satisfies 
\begin{gather*}
\frac{dL_n}{dt}\Big|_\eqref{limsys0} = \lambda_n t^{-\frac{n}{2q}}\left( 
  |  \eta|u^2
+ |\nu_n|v^2 
+\mathcal O(\Delta^3)+\mathcal O(\Delta^2)\mathcal O(t^{-\frac{1}{2q}})\right) 
\end{gather*}
as $\Delta\to 0$ and $t\to\infty$. Therefore, there exists $0<\Delta_1\leq \Delta_0$ and $t_2\geq t_1$ such that 
\begin{gather}\label{dLn00}
\begin{split} 
			&\frac{dL_n}{dt}\Big|_\eqref{limsys0}  \geq \gamma_n t^{-\frac{n}{2q}} \left(|\nu_n|u^2+|\eta|v^2\right) \geq  \gamma_n t^{-\frac{2n-1}{2q}} L_n\geq 0 \quad \text{if} \quad \lambda_n>0, \\
			&\frac{dL_n}{dt}\Big|_\eqref{limsys0} \leq -\gamma_n t^{-\frac{n}{2q}} \left(|\nu_n|u^2+|\eta|v^2\right)  \leq -\gamma_n t^{-\frac{2n-1}{2q}} L_n \leq 0 \quad \text{if} \quad \lambda_n<0
		\end{split}
\end{gather}
for all $(u,v,t)\in\mathbb R^3$ such that $\Delta\leq \Delta_1$ and $t\geq t_2$ with $\gamma_n=|\lambda_n|/2>0$. Integrating \eqref{dLn00}, we obtain the following inequalities:
\begin{align*}
&|\nu_n|u^2(t)+|\eta|v^2(t)\geq C t^{-\frac{n-1}{2q}} \exp \left\{ \frac{2q \gamma_n}{2q-2n+1}\left(t^{1-\frac{2n-1}{2q}}-t_2^{1-\frac{2n-1}{2q}}\right)\right\}\quad \text{if} \quad \lambda_n>0,\\
&|\nu_n|u^2(t)+|\eta|v^2(t)\leq 4C \exp \left\{ -\frac{2q \gamma_n}{2q-2n+1}\left(t^{1-\frac{2n-1}{2q}}-t_2^{1-\frac{2n-1}{2q}}\right)\right\}\quad \text{if} \quad \lambda_n<0
\end{align*}
with a positive parameter $C=L_n(u(t_2),v(t_2),t_2)>0$. Thus, if $\lambda_n>0$ and $n\leq q$, the equilibrium $(0,\psi_0)$ of system \eqref{limsys} is unstable. If $n\leq q$ and $\lambda_n<0$, the equilibrium is asymptotically stable. Finally, if $n>q$ and $\lambda_n<0$, the equilibrium is (non-asymptotically) stable. 
\end{proof}

\begin{proof}[Proof of Lemma~\ref{Lem1}]
Substituting the series \eqref{assol} into \eqref{trsys} and equating the terms of like powers of $t$ yield the chain of linear equations for the coefficients $\varrho_k$, $\phi_k$
\begin{gather}\label{rpk}
\begin{pmatrix}  
	\eta & 0 \\
\displaystyle	 \lambda_n & \nu_n 
\end{pmatrix} 
\begin{pmatrix} 
 \varrho_k \\ \phi_k 
\end{pmatrix} 
= 
\begin{pmatrix} 
 \mathfrak  F_k \\ \mathfrak  G_k 
\end{pmatrix},
\end{gather}
where $\mathfrak  F_k$, $\mathfrak G_k$ are expressed through $\varrho_1,\phi_1, \dots, \varrho_{k-1}, \phi_{k-1}$. For instance,
\begin{align*}
&\mathfrak  F_1=-\Omega_m(0,\psi_0), \\ 
& \mathfrak  G_1=-\Lambda_{n+1}(0,\psi_0),\\
&\mathfrak  F_2=-\Omega_{m+1}(0,\psi_0)-\left(\varrho_1\partial_\rho +\phi_1 \partial_\psi\right)\Omega_{m}(0,\psi_0), \\ 
&\mathfrak  G_2=-\Lambda_{n+2}(0,\psi_0)-\sum_{i+j=2}\left(\varrho_i\partial_\rho +\phi_i \partial_\psi\right)\Lambda_{n+j}(0,\psi_0)-\frac{1}{2}\left(\varrho_1^2\partial^2_\rho+2\varrho_1\phi_1\partial_\rho\partial_\psi+\phi_1^2\partial^2_\psi\right)\Lambda_{n}(0,\psi_0).
\end{align*}
Since $\nu_n\eta\neq 0$, we see that system \eqref{rpk} is solvable. 

To prove the existence of a solution of system \eqref{trsys} with such asymptotic behaviour, consider the following functions:
\begin{gather}\label{sumM}
\varrho_{\ast,M}(t)\equiv\sum_{k=1}^{n+M+1} t^{-\frac{k+m-2}{2q}} \varrho_k, \quad 
\phi_{\ast,M}(t)\equiv\psi_0+\sum_{k=1}^{n+M+1} t^{-\frac{k}{2q}}\phi_k
\end{gather}
with some $M\in\mathbb Z_+$. By construction, 
\begin{gather}
	\label{Zeq}
		\begin{split}
&Z_{\varrho}(t)\equiv  \varrho_{\ast,M}'(t)-\hat\Lambda_N(\varrho_{\ast,M}(t),\phi_{\ast,M}(t),t)=\mathcal O\left(t^{-\frac{2n+M+2}{2q}}\right),\\
&Z_{\phi}(t)\equiv \phi_{\ast,M}'(t)-\hat\Omega_N(\varrho_{\ast,M}(t),\phi_{\ast,M}(t),t)=\mathcal O\left(t^{-\frac{n+m+M+1}{2q}}\right), \quad t\to\infty.
\end{split}
\end{gather}
Substituting
\begin{gather} \label{subsM}
\varrho(t)=\varrho_{\ast,M}(t)+t^{-\frac{M}{2q}} u(t), \quad 
\phi(t)=\phi_{\ast,M}(t)+t^{-\frac{M}{2q}} v(t)
\end{gather}
into \eqref{trsys}, we obtain a perturbed near-Hamiltonian system
\begin{gather}\label{uvsys}
\frac{du}{dt}=-\partial_v \mathcal  H_M(u,v,t) + \xi_M(t), \quad 
\frac{dv}{dt}=\partial_u \mathcal  H_M(u,v,t)+\Upsilon_M(u,v,t) + \zeta_M(t), 
\end{gather}
with  
\begin{align*}
&\mathcal H_M(u,v,t) \equiv \int\limits_0^u \mathcal G_M(w,0,t)\,dw - \int\limits_0^v \mathcal F_M(u,w,t)\,dw, \\ 
&\Upsilon_M(u,v,t)  \equiv  \int\limits_0^v \left(\partial_u\mathcal F(u,w,t)+\partial_v\mathcal G(u,w,t)\right)\,dw
\end{align*}
and perturbations
\begin{gather*}
 \xi_M(t)\equiv - t^{\frac{M}{2q}}Z_\varrho(t), \quad \zeta_M(t)\equiv - t^{\frac{M}{2q}}Z_\phi(t),
\end{gather*}
where 
\begin{align*}
\mathcal  F_M(u,v,t)&\equiv t^{\frac{M}{2q}}\left(\hat\Lambda_N(\varrho_{\ast,M}(t)+t^{-\frac{M}{2q}} u,\phi_{\ast,M}(t)+t^{-\frac{M}{2q}} v,t)-\hat\Lambda_N(\varrho_{\ast,M}(t),\phi_{\ast,M}(t),t)\right) + \frac{M}{2q} t^{-1} u, \\
\mathcal  G_M(u,v,t)&\equiv t^{\frac{M}{2q}}\left(\hat\Omega_N(\varrho_{\ast,M}(t)+t^{-\frac{M}{2q}} u,\phi_{\ast,M}(t)+t^{-\frac{M}{2q}} v,t)-\hat\Omega_N(\varrho_{\ast,M}(t),\phi_{\ast,M}(t),t)\right) + \frac{M}{2q} t^{-1} v.
\end{align*}
It follows from  \eqref{asn}, \eqref{aszero} and \eqref{Zeq} that
\begin{gather}\label{HU1}
\begin{split}
& \mathcal  H_M(u,v,t) =\left\{t^{-\frac{1}{2q}}\frac{\eta u^2}{2}-t^{-\frac{n}{2q}}\left(\lambda_n u v+ \frac{\nu_n v^2}{2}\right)\right\}\left(1+\mathcal O(t^{-\frac{1}{2q}})\right),\\
& \Upsilon_M(u,v,t)=v\left(\lambda_n t^{-\frac{n}{2q}}+\omega_m t^{-\frac{m}{2q}}\right) \left(1+\mathcal O(t^{-\frac{1}{2q}})\right),\\
& \xi_M(t)=\mathcal O(t^{-\frac{2n+2}{2q}}),\quad 
\zeta_M(t)=\mathcal O(t^{-\frac{n+m+1}{2q}})
\end{split}
\end{gather}
as $\Delta=\sqrt{u^2+v^2}\to 0$ and $t\to\infty$. Our goal is to show that there is a solution of system \eqref{uvsys} such that $u(t)=\mathcal O(1)$ and $v(t)=\mathcal O(1)$ as $t\to\infty$. This will ensure the existence of a solution to system \eqref{trsys} with asymptotic expansion \eqref{assol}. The proposed method is based on the stability analysis and on the construction of suitable Lyapunov functions. Note that a similar approach to justifying the asymptotics was used in~\cite{LK15}.

Note that if $\xi_M(t)\equiv \zeta_M(t)\equiv 0$, then system \eqref{uvsys} has the equilibrium $(0,0)$. Let us prove the stability of the near-Hamiltonian system with respect to the time-decaying perturbations $\xi_M(t)$ and $\zeta_M(t)$.

Consider first the case $1\leq n<m$. If $n=1$, we use $\mathcal L_M(u,v,t)\equiv L_1(u,v)$ defined by \eqref{LF1} and \eqref{chi1} as a Lyapunov function candidate for system \eqref{uvsys}. If $n>1$, we use
\begin{gather*}%\label{LFn}
\begin{split}
\mathcal L_M(u,v,t)\equiv &  ({\hbox{\rm sgn}}\, \eta)t^{\frac{n}{2q}} \mathcal H_M(u,v,t)  + t^{-\frac{n-1}{2q}}\lambda_n \mathcal K(u,v),
\end{split}
\end{gather*}
with $\mathcal K(u,v)\equiv {\lambda_n  v^2}/{2|\eta|}+  ({\hbox{\rm sgn}}\, \eta) u v$.
We see that
\begin{gather}\label{Lnas}
\mathcal L_M(u,v,t) = t^{\frac{n-1}{2q}}\frac{|\eta|u^2}{2}+ \frac{|\nu_n|v^2}{2}-({\hbox{\rm sgn}}\, \eta) \lambda_n uv +\mathcal O(\Delta^2)\mathcal O(t^{-\frac{1}{2q}}), \quad \Delta\to 0, \quad t\to\infty.
\end{gather}
Hence, there exist $\Delta_1>0$ and $t_1\geq  t_0 $ such that  
\begin{gather}\label{Lnineq1}
L_- \Delta^2 \leq\mathcal L_M(u,v,t)\leq t^{\frac{n-1}{2q}}L_+\Delta^2
\end{gather}
 for all $(u,v,t)\in\mathbb R^3$ such that $\Delta\leq \Delta_1$ and $t\geq t_1$ with some $L_\pm={\hbox{\rm const}}>0$. The derivative of $\mathcal L_M(u,v,t)$ with respect to $t$ along the trajectories of the system is given by 
\begin{gather}\label{dLn}
\frac{d \mathcal L_M}{dt}\Big|_\eqref{uvsys}\equiv \mathcal  D_{M,1}(u,v,t)+\mathcal  D_{M,2}(u,v,t),
\end{gather}
where $\mathcal  D_{M,1}\equiv \left(\partial_t -\partial_u \mathcal  H_M \partial_u +(\partial_v\mathcal H_M +\Upsilon_M) \partial_v \right)\mathcal L_M$ and $\mathcal  D_{M,2}\equiv \left(\xi_M \partial_u +\zeta_M \partial_v \right)\mathcal L_M$. It can easily be checked that 
\begin{align*}
& \mathcal  D_{M,1}(u,v,t)\leq -t^{-\frac{1}{2q}} \left(A_1 u^2 + B_1 v^2 +\mathcal O(\Delta^2)\mathcal O(t^{-\frac{1}{2q}})\right)    \text{if} \quad  n=1,\\
& \mathcal  D_{M,1}(u,v,t)= -  t^{-\frac{n}{2q}} |\lambda_n| \left(|\eta|u^2 + |\nu_n| v^2 +\mathcal O(\Delta^2)\mathcal O(t^{-\frac{1}{2q}})\right)    \text{if} \quad n>1,
\end{align*}
and $\mathcal  D_{M,2}(u,v,t) = \mathcal O(\Delta)\mathcal O(t^{-\frac{n+3}{2q}})$ as $\Delta\to 0$ and $t\to\infty$, where the positive parameters $A_1$ and $B_1$ are defined by \eqref{A1B1}. It follows that there exist $\Delta_2\leq \Delta_1$ and $t_2\geq t_1$ such that 
\begin{gather*}
\mathcal  D_{M,1}(u,v,t)\leq - t^{-\frac{n}{2q}} \gamma_n \Delta^2, \quad 
\mathcal  D_{M,2}(u,v,t)\leq t^{-\frac{n+1}{2q}} C \Delta
\end{gather*}
for all $(u,v,t)\in\mathbb R^3$ such that $\Delta\leq \Delta_2$ and $t\geq t_2$, where $C={\hbox{\rm const}}>0$, $\gamma_1=\min\{A_1,B_1\}/2$ and $\gamma_n=|\lambda_n|\min\{|\eta|,|\nu_n|\}/2$. Therefore, for all $\epsilon\in (0,\Delta_2)$ there exist
\begin{gather*}
\delta_\epsilon=\frac{2C }{\gamma_n} t_\epsilon^{-\frac{1}{2q}}, \quad t_\epsilon=\max\left\{t_2,\left(\frac{4 C }{\gamma_n \epsilon}\right)^{2q}\right\}
\end{gather*}
such that 
\begin{gather*}
\frac{d\mathcal L_M}{dt}\Big|_\eqref{trsys}\leq 	t^{-\frac{n}{2q}}\left(-\gamma_n+C\delta_\epsilon^{-1}t_\epsilon^{-\frac{1}{2q}}\right)\Delta^2 \leq 0
\end{gather*}
for all $(u,v,t)\in\mathbb R^3$ such that $\delta_\epsilon\leq \Delta\leq \epsilon$ and $t\geq t_\epsilon$. Combining this with \eqref{Lnineq1}, we see that any solution of system \eqref{uvsys} with initial data $\sqrt{u^2(t_\epsilon)+v^2(t_\epsilon)}\leq \delta$, where $\delta=\max\{\delta_\epsilon,\epsilon\sqrt{L_-/L_+}\}$, cannot exit from the domain $\{(u,v)\in\mathbb R^2: \Delta\leq \epsilon\}$ as $t\geq t_\epsilon$. It follows from \eqref{subsM} that for all $M\in\mathbb Z_+$ the trajectories of system \eqref{trsys} starting close to $(0,\psi_0)$ satisfy the estimates $\varrho(t)=\varrho_{\ast,M}(t)+\mathcal O(t^{-\frac{M}{2q}})$, $\phi(t)=\phi_{\ast,M}(t)+\mathcal O(t^{-\frac{M}{2q}})$ as $t\to\infty$. Thus, there exists the solution $\varrho_\ast(t)$, $\phi_\ast(t)$ of system \eqref{trsys} with asymptotics \eqref{assol}. 

Now let $n\geq 2$. Using 
\begin{gather*}
\mathcal L_M(u,v,t)\equiv 
\begin{cases}
({\hbox{\rm sgn}}\, \eta)t^{\frac{n}{2q}}\mathcal H_M(u,v,t)+ t^{-\frac{n-1}{2q}} (\lambda_n+\omega_n)\mathcal K(u,v), & n=m,\\
({\hbox{\rm sgn}}\, \eta) t^{\frac{n}{2q}}\mathcal H_M(u,v,t) + t^{-\frac{m-1}{2q}} \omega_m \mathcal K(u,v), & n>m,
\end{cases}
\end{gather*}
as a Lyapunov function candidate for system \eqref{uvsys}, we obtain \eqref{Lnas} and \eqref{dLn}, where
\begin{align*}
&\mathcal  D_{M,1}(u,v,t)=
\begin{cases}
-t^{-\frac{n}{2q}}|\lambda_n+\omega_n|\left(|\eta|u^2 + |\nu_n| v^2 +\mathcal O(\Delta^2)\mathcal O(t^{-\frac{1}{2q}})\right), & n=m,\\
-t^{-\frac{m}{2q}}|\omega_m|\left(|\eta|u^2 + |\nu_n| v^2 +\mathcal O(\Delta^2)\mathcal O(t^{-\frac{1}{2q}})\right), & n>m,
\end{cases}\\
& \mathcal  D_{M,2}(u,v,t) = \begin{cases}
\mathcal O(\Delta) \mathcal O(t^{-\frac{n+1}{2q}}), & n=m,\\
\mathcal O(\Delta)\mathcal O(t^{-\frac{m+1}{2q}}), & n>m,
\end{cases}
\end{align*}
as $\Delta\to 0$ and $t\to\infty$. Then, repeating the arguments as given above proves the existence of the solution with asymptotics \eqref{assol}.
 
To prove the stability of the constructed solution consider the substitution \eqref{subsM} with $\varrho_\ast(t)$, $\phi_\ast(t)$ instead of $\varrho_{\ast,M}(t)$, $\phi_{\ast,M}(t)$ and with some $M\in\mathbb Z_+$. In this case, we obtain system \eqref{uvsys} with 
\begin{align*}
& \mathcal  F_M(u,v,t)\equiv t^{\frac{M}{2q}}\left(\hat\Lambda_N(\varrho_{\ast}(t)+t^{-\frac{M}{2q}} u,\phi_{\ast}(t)+t^{-\frac{M}{2q}} v,t)-\hat\Lambda_N(\varrho_{\ast}(t),\phi_{\ast}(t),t)\right) + \frac{M}{2q} t^{-1} u, \\
& \mathcal  G_M(u,v,t)\equiv t^{\frac{M}{2q}}\left(\hat\Omega_N(\varrho_{\ast}(t)+t^{-\frac{M}{2q}} u,\phi_{\ast}(t)+t^{-\frac{M}{2q}} v,t)-\hat\Omega_N(\varrho_{\ast}(t),\phi_{\ast}(t),t)\right) + \frac{M}{2q} t^{-1} v,
\end{align*} 
and $ \xi_M(t)\equiv  \zeta_M(t)\equiv 0$. Then, repeating the arguments as given 
above and using the constructed Lyapunov functions, we get $d\mathcal L_M/dt\leq - t^{-\frac{2n-1}{2q}}D_n \mathcal L_M$ for all $(u,v,t)\in\mathbb R^3$ such that $\Delta\leq \Delta_3$, $t\geq t_3$ with some $\Delta_3\leq \Delta_1$, $t_3\geq t_1$ and $D_n=\gamma_n/L_+>0$. Integrating this inequality and taking into account \eqref{Lnineq1}, we obtain asymptotic stability of the solution $\varrho_\ast(t)$, $\phi_\ast(t)$ if $n\leq q$ and (non-asymptotic) stability if $n>q$.
\end{proof}

\begin{proof}[Proof of Lemma~\ref{Lem2}]
The asymptotic series are constructed in the same way as in the proof of Lemma~\ref{Lem1}.
Consider the functions $\varrho_{\ast,M}(t)$, $\phi_{\ast,M}(t)$ defined by \eqref{sumM}. Substituting 
\begin{gather}\label{subsM2}
\varrho(t)=\varrho_{\ast,M}(t)+t^{-\frac{M-h}{2q}} u(t), \quad 
\phi(t)=\phi_{\ast,M}(t)+t^{-\frac{M-h}{2q}}v(t)
\end{gather}
with $M>h$ into equations \eqref{trsys}, we get perturbed near-Hamiltonian system \eqref{uvsys}, where $\mathcal H_M(u,v,t)$ and $\Upsilon_M(u,v,t)$ are defined by \eqref{HU1} with
\begin{gather*}
	\mathcal  F_M  \equiv  t^{\frac{M-h}{2q}}\left(\hat\Lambda_N(\varrho_{\ast,M}(t)+t^{-\frac{M-h}{2q}} u,\phi_{\ast,M}(t)+t^{-\frac{M-h}{2q}} v,t)-\hat\Lambda_N(\varrho_{\ast,M}(t),\phi_{\ast,M}(t),t)\right) + \frac{M-h}{2q} t^{-1} u, \\
	\mathcal  G_M \equiv  t^{\frac{M-h}{2q}}\left(\hat\Omega_N(\varrho_{\ast,M}(t)+t^{-\frac{M-h}{2q}} u,\phi_{\ast,M}(t)+t^{-\frac{M-h}{2q}} v,t)-\hat\Omega_N(\varrho_{\ast,M}(t),\phi_{\ast,M}(t),t)\right) + \frac{M-h}{2q} t^{-1} v,
\end{gather*}
and the perturbations have the following form:
\begin{gather*}
	\xi_M(t) \equiv - t^{\frac{M-h}{2q}}Z_\varrho(t), \quad
	\zeta_M(t) \equiv - t^{\frac{M-h}{2q}}Z_\phi(t)
\end{gather*}
with functions $Z_\varrho(t)$ and $Z_\phi(t)$ defined by \eqref{Zeq}.
It follows easily that 
\begin{align*}
&\mathcal H_M(u,v,t)= \left\{t^{-\frac{1}{2q}}\frac{\eta u^2}{2}-t^{-\frac{n}{2q}}\left(\lambda_n u v+ \frac{\nu_n v^2}{2}\right)\right\}(1+\mathcal O(t^{-\frac{1}{2q}})),\\
&\Upsilon_M(u,v,t)=t^{-\frac{h}{2q}} d_h v(1+\mathcal O(t^{-\frac{1}{2q}})), \\ 
&\xi_M(t)=\mathcal O(t^{-\frac{h+2n+2}{2q}}), \quad 
\zeta_M(t)=\mathcal O(t^{-\frac{h+n+m+1}{2q}})
\end{align*}
as $\Delta=\sqrt{u^2+v^2}\to 0$ and $t\to\infty$.  

Note that $\partial_u \mathcal H_M(0,0,t)\equiv \partial_v \mathcal H_M(0,0,t) \equiv \Upsilon_M(0,0,t)\equiv 0$.  Let us prove the stability of the system with respect to the non-vanishing perturbations $\xi_M(t)$ and $\zeta_M(t)$ (see~\cite[Ch.~9]{Halil}). 

Consider a Lyapunov function candidate in the following form:
\begin{gather*}
\mathcal L_M(u,v,t)\equiv 
\begin{cases}
({\hbox{\rm sgn}}\, \eta)t^{\frac{1}{2q}} \mathcal H_M(u,v,t)+t^{-\frac{h-1}{2q}}\frac{d_h({\hbox{\rm sgn}}\, \eta)}{2}u v , & h>n=1,\\
({\hbox{\rm sgn}}\, \eta)t^{\frac{n}{2q}} \mathcal H_M(u,v,t)+t^{-\frac{h-1}{2q}} d_h ({\hbox{\rm sgn}}\, \eta)u v, & h>n>1,\\
({\hbox{\rm sgn}}\, \eta)t^{\frac{n}{2q}} \mathcal H_M(u,v,t) +t^{-\frac{h-1}{2q}} d_h\left\{({\hbox{\rm sgn}}\, \eta)u v+\frac{\lambda_n v^2}{2|\eta|}\right \}  , & h\leq n.
\end{cases}
\end{gather*}
We see that there exist $\Delta_1>0$ and $t_1\geq t_0$ such that the estimate \eqref{Lnineq1} holds for all $(u,v,t)\in\mathbb R^3$ such that $\Delta\leq \Delta_1$ and $t\geq t_1$ with some $L_\pm={\hbox{\rm const}}>0$. 
The derivative of $\mathcal L_M(u,v,t)$ with respect to $t$ along the trajectories of the system is given by \eqref{dLn}, where $\mathcal  D_{M,1}\equiv \left(\partial_t - \partial_v \mathcal H_M \partial_u +(\partial_u \mathcal H_M+\Upsilon_M)  \partial_v \right)\mathcal L_M$ and $\mathcal  D_{M,2}\equiv \left(\xi_M\partial_u +\zeta_M\partial_v \right) \mathcal L_M$. 
We see that 
\begin{gather*}
	\mathcal  D_{M,1}=
		\begin{cases}
			t^{-\frac{h}{2q}}\frac{d_n}{2}(|\eta|u^2+|\nu_1|v^2)(1 +\mathcal O(t^{-\frac{1}{q}})), & n=1, \\
			t^{-\frac{h}{2q}}d_n(|\eta|u^2+|\nu_1|v^2)(1 +\mathcal O(t^{-\frac{1}{q}})), & n\neq 1
		\end{cases}
\end{gather*}
and $\mathcal  D_{M,2}(u,v,t) = \mathcal O(t^{-\frac{h+n+3}{2q}})\mathcal O(\Delta)$ as $\Delta\to 0$ and $t\to\infty$. It follows that there exist  $\Delta_2\leq \Delta_1$ and $t_2\geq t_1$ such that 
$\mathcal  D_{M,1}(u,v,t)\leq - t^{-\frac{h}{2q}} \gamma_h \Delta^2$ and
$\mathcal  D_{M,2}(u,v,t)\leq t^{-\frac{h+1}{2q}} C \Delta$
for all $(u,v,t)\in\mathbb R^3$ such that $\Delta\leq \Delta_2$ and  $t\geq t_2$, where $C={\hbox{\rm const}}>0$,  
and $\gamma_h=|d_h|/4>0$. Hence, for all $\epsilon\in (0,\Delta_2)$ there exist
\begin{gather*}
\delta_\epsilon=\frac{2C }{\gamma_h} t_\epsilon^{-\frac{1}{2q}}, \quad t_\epsilon=\max\left\{t_2,\left(\frac{4 C }{\gamma_h\epsilon}\right)^{2q}\right\}
\end{gather*}
such that 
\begin{gather*}
\frac{d\mathcal L_M}{dt}\Big|_\eqref{uvsys}\leq t^{-\frac{h}{2q}}\left(-\gamma_h+C\delta_\epsilon^{-1}t_\epsilon^{-\frac{1}{2q}}\right)\Delta^2 \leq 0
\end{gather*}
for all $\delta_\epsilon\leq \Delta\leq \epsilon$ and $t\geq t_\epsilon$. Taking into account \eqref{Lnineq1}, we see that solutions of system \eqref{uvsys} with initial data $\sqrt{u^2(t_\epsilon)+v^2(t_\epsilon)}\leq \delta$ and $\delta=\max\{\delta_\epsilon,\epsilon\sqrt{L_-/L_+}\}$  cannot exit from the domain $\{(u,v)\in\mathbb R^2: \Delta\leq \epsilon\}$ as $t\geq t_\epsilon$.  Thus, for all $M>h$ the solutions of system \eqref{trsys} starting close to $(0,\psi_0)$ satisfy the estimates $\varrho(t)=\varrho_{\ast,M}(t)+\mathcal O(t^{-\frac{M-h}{2q}})$, $\phi(t)=\phi_{\ast,M}(t)+\mathcal O(t^{-\frac{M-h}{2q}})$ as $t\to\infty$. This ensures the existence of a particular solution $\varrho_\ast(t)$, $\phi_\ast(t)$ of system \eqref{trsys} with asymptotic expansion \eqref{assol}. 

To prove the stability of the constructed solution consider the substitution \eqref{subsM2} with $\varrho_\ast(t)$, $\phi_\ast(t)$ instead of $\varrho_{\ast,M}(t)$, $\phi_{\ast,M}(t)$ and some integer $M>h$. In this case, we obtain system \eqref{uvsys} with 
\begin{align*}
& \mathcal  F_M(u,v,t)\equiv t^{\frac{M-h}{2q}}\left(\hat\Lambda_N(\varrho_{\ast}(t)+t^{-\frac{M-h}{2q}} u,\phi_{\ast}(t)+t^{-\frac{M-h}{2q}} v,t)-\hat\Lambda_N(\varrho_{\ast}(t),\phi_{\ast}(t),t)\right) + \frac{M-h}{2q} t^{-1} u, \\
& \mathcal  G_M(u,v,t)\equiv t^{\frac{M-h}{2q}}\left(\hat\Omega_N(\varrho_{\ast}(t)+t^{-\frac{M-h}{2q}} u,\phi_{\ast}(t)+t^{-\frac{M-h}{2q}} v,t)-\hat\Omega_N(\varrho_{\ast}(t),\phi_{\ast}(t),t)\right) + \frac{M-h}{2q} t^{-1} v,
\end{align*} 
and $ \xi_M(t)\equiv  \zeta_M(t)\equiv 0$. Then, repeating the arguments as given 
above and using the constructed Lyapunov functions $\mathcal L_M(u,v,t)$, we obtain the inequality 
$d\mathcal L_M/dt\leq - t^{-\frac{h+n-1}{2q}}D_h \mathcal L_M$ for all $(u,v,t)\in\mathbb R^3$ such that $\Delta\leq \Delta_3$, $t\geq t_3$ with some $\Delta_3\leq \Delta_1$, $t_3\geq t_1$ and $D_h=\gamma_h/L_+>0$. Integrating the inequality with respect to $t$ and taking into account \eqref{Lnineq1}, we obtain the asymptotic stability of the solution $\varrho_\ast(t)$, $\phi_\ast(t)$ if $h+n\leq 2q+1$, and the (non-asymptotic) stability if $h+n>2q+1$.
\end{proof}

\section{Analysis of the full system}\label{sec5}
\begin{proof}[Proof of Theorem~\ref{Th2}]
Substituting $\varrho(t)=\varrho_{\ast}(t)+u(t)$, $\phi(t)=\phi_{\ast}(t)+  v(t)$ into \eqref{rhopsi}, we obtain a perturbed near-Hamiltonian system
\begin{gather}\label{uvsys1}
\frac{du}{dt}=-\partial_v \mathcal  H(u,v,t) + \mathcal   P_N(u,v,t), \quad 
\frac{dv}{dt}=\partial_u \mathcal  H(u,v,t) + \Upsilon(u,v,t)+ \mathcal   Q_N(u,v,t), 
\end{gather}
with the Hamiltonian
\begin{gather*}
\mathcal H(u,v,t)   \equiv \int\limits_0^u \mathcal G(w,0,t)\,dw - \int\limits_0^v \mathcal F(u,w,t)\,dw,
\end{gather*}
and perturbations
\begin{align*} 
\Upsilon(u,v,t)  &\equiv  \int\limits_0^v \left(\partial_u\mathcal F(u,w,t)+\partial_v\mathcal G(u,w,t)\right)\,dw,\\
\mathcal  F(u,v,t)& \equiv \hat\Lambda_N(\varrho_{\ast}(t)+u,\phi_{\ast}(t)+v,t)-\hat\Lambda_N(\varrho_{\ast}(t),\phi_{\ast}(t),t),\\
\mathcal  G(u,v,t)& \equiv \hat\Omega_N(\varrho_{\ast}(t)+u,\phi_{\ast}(t)+v,t)-\hat\Omega_N(\varrho_{\ast}(t),\phi_{\ast}(t),t),\\
\mathcal  P_N(u,v,t)&\equiv  \tilde \Lambda_N(\varrho_{\ast}(t)+u,\phi_{\ast}(t)+v,S(t),t), \\
\mathcal  Q_N(u,v,t)&\equiv \tilde \Omega_N(\varrho_{\ast}(t)+u,\phi_{\ast}(t)+v,S(t),t).
\end{align*}
It follows from \eqref{tildeLO}, \eqref{aszero} and \eqref{assol} that
\begin{gather}\label{HUPQest}
\begin{split}
 \mathcal  H(u,v,t)&=\left\{t^{-\frac{1}{2q}}\frac{\eta u^2}{2}-t^{-\frac{n}{2q}}\left(\lambda_n u v+ \frac{\nu_n v^2}{2}\right)\right\}\left(1+\mathcal O(\Delta)+\mathcal O(t^{-\frac{1}{2q}})\right),\\
 \Upsilon(u,v,t)&=v\left(\lambda_n t^{-\frac{n}{2q}}+\omega_m t^{-\frac{m}{2q}}\right) \left(1+\mathcal O(\Delta)+\mathcal O(t^{-\frac{1}{2q}})\right),\\
 \mathcal P_N(u,v,t)&=\mathcal O(t^{-\frac{N+1}{2q}}), \\
 \mathcal Q_N(u,v,t)&=\mathcal O(t^{-\frac{N+1}{2q}})
\end{split}
\end{gather}
as $\Delta=\sqrt{u^2+v^2}\to 0$ and $t\to\infty$. Note that $\partial_u\mathcal  H(0,0,t)\equiv\partial_v\mathcal  H(0,0,t)\equiv \Upsilon(0,0,t)\equiv 0$, while the functions $\mathcal P_N(u,v,t)$ and $\mathcal Q_N(u,v,t)$ do not preserve the equilibrium $(0,0)$ and can be considered as external perturbations. Let us prove the stability of the equilibrium in the perturbed system~\cite[Ch.~9]{Halil}.

Consider a Lyapunov function candidate in the form
\begin{gather}\label{LFTh}
\mathcal L(u,v,t)\equiv 
\begin{cases}
({\hbox{\rm sgn}}\, \eta)t^{\frac{1}{2q}} \mathcal H(u,v,t)+(\chi_1+\lambda_1  {\hbox{\rm sgn}}\, \eta ) u v , & n=1,\\
({\hbox{\rm sgn}}\, \eta)t^{\frac{n}{2q}} \mathcal H(u,v,t)  + t^{-\frac{n-1}{2q}}\lambda_n \mathcal K(u,v), & 1<n<m,\\
({\hbox{\rm sgn}}\, \eta)t^{\frac{n}{2q}} \mathcal H(u,v,t)  + t^{-\frac{n-1}{2q}}(\lambda_n+\omega_n)\mathcal K(u,v), & 1<n=m,\\
({\hbox{\rm sgn}}\, \eta)t^{\frac{n}{2q}} \mathcal H(u,v,t)  + t^{-\frac{m-1}{2q}}\omega_m \mathcal K(u,v), & n>m\geq 2,
\end{cases}
\end{gather}
with $\mathcal K(u,v)\equiv {\lambda_n v^2}/{|2\eta|}+ ({\hbox{\rm sgn}}\, \eta) u v$ and the parameter $\chi_1$ defined by \eqref{chi1}. Note that if $n>1$,
\begin{gather*}
\mathcal L(u,v,t) = t^{\frac{n-1}{2q}}\frac{|\eta|u^2}{2}+ \frac{|\nu_n|v^2}{2}-({\hbox{\rm sgn}}\, \eta) \lambda_n uv +\mathcal O(\Delta^3)+\mathcal O(\Delta^2)\mathcal O(t^{-\frac{1}{2q}})
\end{gather*}
as $\Delta\to 0$ and $t\to\infty$. It follows that there exist $\Delta_1>0$ and $t_1\geq t_0$ such that  $\mathcal L(u,v,t)$ satisfies the inequalities \eqref{Lnineq1} for all $(u,v,t)\in\mathbb R^3$ such that $\Delta\leq \Delta_1$ and $t\geq t_1$ with some $L_\pm={\hbox{\rm const}}>0$. 
The derivative of $\mathcal L(u,v,t)$ with respect to $t$ along the trajectories of system \eqref{uvsys1} is given by 
\begin{gather} \label{dLuv1}
\frac{d\mathcal L}{dt}\Big|_\eqref{uvsys1}\equiv \mathcal  D_{1}(u,v,t)+\mathcal  D_{2,N}(u,v,t),
\end{gather}
where $\mathcal  D_{1}\equiv \left(\partial_t -\partial_u \mathcal  H \partial_u +(\partial_v\mathcal  H+\Upsilon) \partial_v \right)\mathcal L$ and $\mathcal  D_{2,N}\equiv \left(\mathcal   P_N\partial_u +\mathcal   Q_N\partial_v \right)\mathcal L$. We see that 
\begin{gather*}
\mathcal  D_1=
\begin{cases}
-t^{-\frac{1}{2q}}\left( (|\lambda_1|-|\chi_1|) |\eta|u^2+|\chi_1 \nu_1|v^2+\chi_1 |\lambda_1|  uv +\mathcal O(\Delta^3)\right)+\mathcal O(\Delta^2)\mathcal O(t^{-\frac{1}{q}}), & n=1, \\
- t^{-\frac{n}{2q}} |\lambda_n|\left(|\eta|u^2 + |\nu_n| v^2 +\mathcal O(\Delta^3)\right)+\mathcal O(\Delta^2)\mathcal O(t^{-\frac{n+1}{2q}}), & 1<n<m, \\
- t^{-\frac{n}{2q}} |\lambda_n+\omega_n|\left(|\eta|u^2 + |\nu_n| v^2 +\mathcal O(\Delta^3)\right)+\mathcal O(\Delta^2)\mathcal O(t^{-\frac{n+1}{2q}}), & n=m, \\
- t^{-\frac{m}{2q}} |\omega_m|\left(|\eta|u^2 + |\nu_n| v^2 +\mathcal O(\Delta^3)\right)+\mathcal O(\Delta^2)\mathcal O(t^{-\frac{m+1}{2q}}), & n>m,
\end{cases}
\end{gather*}
and $\mathcal  D_{2,N}(u,v,t) = \mathcal O(\Delta)\mathcal O(t^{-\frac{N-n+2}{2q}})$ as $\Delta\to 0$ and $t\to\infty$. It follows that there exist $N_0=\min\{2n-1,n+m-1\}$, $\Delta_2\leq \Delta_1$ and $t_2\geq t_1$ such that 
\begin{gather*}
\mathcal  D_{1}(u,v,t)\leq - t^{-\frac{\ell}{2q}} \gamma \Delta^2, \quad 
\mathcal  D_{2,N}(u,v,t)\leq t^{-\frac{\ell+1}{2q}} C \Delta
\end{gather*}
for all $N\geq N_0$ and $(u,v,t)\in\mathbb R^3$ such that $\Delta\leq \Delta_2$ and $t\geq t_2$, where $C={\hbox{\rm const}}>0$ and $\ell=\min\{n,m\}$. If $n=1$, then $\gamma=\min\{A_1,B_1\}/2$, and if $n>1$, then $\gamma=|d_{n,m}|\min\{|\eta|,|\nu_n|\}/2$. Positive parameters $A_1$ and $B_1$ are defined by \eqref{A1B1}. Hence, for all $\epsilon\in (0,\Delta_2)$ there exist
\begin{gather*}
\delta_\epsilon=\frac{2C }{\gamma} t_\epsilon^{-\frac{1}{2q}}, \quad t_\epsilon=\max\left\{t_2,\left(\frac{4 C }{\gamma\epsilon}\right)^{2q}\right\}
\end{gather*}
such that 
\begin{gather*}
\frac{d\mathcal L}{dt}\Big|_\eqref{uvsys1}\leq t^{-\frac{\ell}{2q}}\left(-\gamma+C\delta_\epsilon^{-1}t_\epsilon^{-\frac{1}{2q}}\right)\Delta^2 \leq 0
\end{gather*}
for all $(u,v,t)\in\mathbb R^3$ such that $\delta_\epsilon\leq \Delta\leq \epsilon$ and $t\geq t_\epsilon$. Taking into account \eqref{Lnineq1}, we see that any solution of system \eqref{uvsys1} with initial data $\sqrt{u^2(t_\epsilon)+v^2(t_\epsilon)}\leq \delta$ and $\delta=\max\{\delta_\epsilon,\epsilon\sqrt{L_-/L_+}\}$  cannot exit from the domain $\{(u,v)\in\mathbb R^2: \Delta\leq \epsilon\}$ as $t\geq t_\epsilon$. 

Thus, returning to the original variables and taking into account Theorem~\ref{Th1} complete the proof.
\end{proof}

\begin{proof}[Proof of Theorem~\ref{Th21}]
Substituting $\varrho(t)=\varrho_{\ast,M}(t)+u(t)$, $\phi(t)=\phi_{\ast,M}(t)+v(t)$ into \eqref{trsys}, we obtain
\begin{gather}\label{uvsys21}
\frac{du}{dt}=-\partial_v H_M(u,v,t) + P_{M,N}(u,v,t), \quad 
\frac{dv}{dt}=\partial_u H_M(u,v,t) + Y_M(u,v,t) + Q_{M,N}(u,v,t), 
\end{gather}
with  
\begin{align*}
H_M(u,v,t) &\equiv \int\limits_0^u \mathcal B_M(w,0,t)\,dw - \int\limits_0^v \mathcal A_M(u,w,t)\,dw,\\
 Y_M(u,v,t)  &\equiv  \int\limits_0^v \left(\partial_u\mathcal A_M(u,w,t)+\partial_v\mathcal B_M(u,w,t)\right)\,dw,\\
\mathcal  A_M(u,v,t) & \equiv \hat\Lambda_N(\varrho_{\ast,M}(t)+u,\phi_{\ast,M}(t)+v,t)-\hat\Lambda_N(\varrho_{\ast,M}(t),\phi_{\ast,M}(t),t),\\
\mathcal  B_M(u,v,t) & \equiv \hat\Omega_N(\varrho_{\ast,M}(t)+ u,\phi_{\ast,M}(t)+v,t)-\hat\Omega_N(\varrho_{\ast,M}(t),\phi_{\ast,M}(t),t),\\
 P_{M,N}(u,v,t) &\equiv  \tilde \Lambda_N(\varrho_{\ast,M}(t)+u,\phi_{\ast,M}(t)+v,S(t),t) - Z_\varrho(t),\\ 
 Q_{M,N}(u,v,t) &\equiv  \tilde \Omega_N(\varrho_{\ast,M}(t)+u,\phi_{\ast,M}(t)+v,S(t),t) - Z_\phi(t),
\end{align*}
where the functions $Z_\varrho(t)$ and $Z_\phi(t)$ are defined by \eqref{Zeq}.
It can easily be checked that 
\begin{align*}
 H_M(u,v,t)&=\left\{t^{-\frac{1}{2q}}\frac{\eta u^2}{2}-t^{-\frac{n}{2q}}\left(\lambda_n u v+ \frac{\nu_n v^2}{2}\right)\right\}\left(1+\mathcal O(\Delta)+\mathcal O(t^{-\frac{1}{2q}})\right),\\
 Y_M(u,v,t)&=v\left(\lambda_n t^{-\frac{n}{2q}}+\omega_m t^{-\frac{m}{2q}}\right) \left(1+\mathcal O(\Delta)+\mathcal O(t^{-\frac{1}{2q}})\right),\\
 P_{M,N}(u,v,t)&=\mathcal O(t^{-\frac{N+1}{2q}})+\mathcal O(t^{-\frac{2n+M+2}{2q}}), \\
 Q_{M,N}(u,v,t)&=\mathcal O(t^{-\frac{N+1}{2q}})+\mathcal O(t^{-\frac{n+m+M+1}{2q}}), \quad \Delta=\sqrt{u^2+v^2}\to 0, \quad t\to\infty.
\end{align*}

Note that if $P_{M,N}(u,v,t)\equiv Q_{M,N}(u,v,t)\equiv 0$, then system \eqref{uvsys21} has the equilibrium $(0,0)$. The functions $P_{M,N}(u,v,t)$ and $Q_{M,N}(u,v,t)$ do not vanish at the equilibrium and play the role of external perturbations in the system. Let us prove the stability of the perturbed system \eqref{uvsys21} by the Lyapunov function method.

Consider the Lyapunov function $\mathcal L(u,v,t)$ in the form \eqref{LFTh}, with $H_M(u,v,t)$ instead of $\mathcal H(u,v,t)$. 
Note that $\mathcal L(u,v,t)$ satisfies \eqref{Lnineq1} for all $(u,v,t)\in\mathbb R^3$ such that $\Delta\leq \Delta_1$ and $t\geq t_1$ with some $L_\pm={\hbox{\rm const}}>0$, $\Delta_1>0$ and $t_1\geq t_0$. The derivative of $\mathcal L(u,v,t)$ with respect to $t$ along the trajectories of system is given by
\begin{gather*} 
\frac{d\mathcal L}{dt}\Big|_\eqref{uvsys21}\equiv \mathcal  D_{1,M}(u,v,t)+\mathcal  D_{2,M,N}(u,v,t),
\end{gather*}
where $\mathcal  D_{1,M}\equiv \left(\partial_t  -\partial_v H_M  \partial_u +(\partial_u H_M+Y_M)\partial_v \right)\mathcal L$ and $\mathcal  D_{2,M,N}\equiv \left(P_{M,N}\partial_u +Q_{M,N} \partial_v \right)\mathcal L$. Note that the following estimates hold:
\begin{align*}
& \mathcal D_{1,M}(u,v,t)\geq t^{-\frac{1}{2q}} (A_1 u^2 + B_1 v^2 +\mathcal O(\Delta^3)+\mathcal O(t^{-\frac{1}{2q}})\mathcal O(\Delta^2)) \quad \text{if} \quad n=1,\\
&\mathcal  D_{1,M}(u,v,t)= |d_{n,m}|  t^{-\frac{\ell}{2q}} \left(|\eta|u^2 + |\nu_n| v^2 +\mathcal O(\Delta^3)+\mathcal O(\Delta^2)\mathcal O(t^{-\frac{1}{2q}})\right)\quad \text{if} \quad n>1, \\
& \mathcal  D_{2,M,N}(u,v,t) =\mathcal O(\Delta)  \mathcal O(t^{-\frac{n+M+1}{2q}})+\mathcal O(\Delta) \mathcal O(t^{-\frac{N-n+2}{2q}})
\end{align*}
as $\Delta\to 0$ and $t\to\infty$, where $\ell=\min\{m,n\}$ and positive parameters $A_1$, $B_1$ are defined by \eqref{A1B1}. 
Hence, there exist $N_0=\min\{2n-1,n+m-1\}$, $C>0$, $t_2\geq t_1$ and $\epsilon\in (0,\Delta_1)$ such that
\begin{gather*}
\frac{d\mathcal L}{dt}\Big|_\eqref{uvsys21}\geq t^{-\frac{\ell}{2q}}\left(\gamma \Delta^2\ -C t^{-\frac{1}{2q}}\Delta\right) 
\end{gather*}
for all $N\geq N_0$ and $(u,v,t)\in\mathbb R^3$ such that $\Delta\leq \epsilon$ and $t\geq t_2$, where $\gamma=\min\{A_1,B_1\}/2$ if $n=1$, and $\gamma=|d_{n,m}|\min\{|\eta|,|\nu_n|\}/2$ if $n>1$. Hence, for all $\delta\in (0,\epsilon)$ there is $t_\ast=\max\{t_2,(2C/|\delta d_{n,m}|)^{2q}\}$ such that 
\begin{gather*}
\frac{d\mathcal L}{dt}\Big|_\eqref{uvsys21}\geq  t^{-\frac{\ell}{2q}} \frac{\gamma}{2}\Delta^2\geq  t^{-\frac{\ell+n-1}{2q}}\tilde \gamma\mathcal L
\end{gather*} 
for all $(u,v,t)\in\mathbb R^3$ such that $\delta\leq \Delta\leq \epsilon$ and $t\geq t_\ast$ with $\tilde \gamma=\gamma/(2L_+)$. Recall that $\ell+n-1<2q$. Then, integrating the last inequality and taking $u(t_\ast)$, $v(t_\ast)$ such that $\sqrt{u^2(t_\ast)+v^2(t_\ast)}=\delta$, we obtain
\begin{gather*}
 u^2(t)+ v^2(t)\geq \frac{\delta^2 L_-}{L_+} t^{-\frac{n-1}{2q}} \exp \left\{ \frac{2q \tilde \gamma}{2q-\ell-n+1}\left(t^{1-\frac{\ell+n-1}{2q}}-t_\ast^{1-\frac{\ell+n-1}{2q}}\right)\right\}, \quad t\geq t_\ast.
\end{gather*}
Hence, there exists $t_e>t_\ast$ such that $ u^2(t_e)+ v^2(t_e)\geq \epsilon^2$. Returning to the  variables  $r(t)$, $\varphi(t)$, we obtain the result of the Theorem.
\end{proof}

\begin{proof}[Proof of Theorem~\ref{Th23}]
Substituting $\varrho(t)=\varrho_{\ast}(t)+u(t)$, $\phi(t)=\phi_{\ast}(t)+  v(t)$ into \eqref{rhopsi}, we obtain system \eqref{uvsys1}. It follows from \eqref{tildeLO}, \eqref{assol} and \eqref{asst} that the functions $\mathcal H(u,v,t)$, $\mathcal P_N(u,v,t)$ and $\mathcal Q_N(u,v,t)$ satisfy \eqref{HUPQest}, while the function $\Upsilon(u,v,t)$ satisfies the following estimate:
\begin{gather*}
\Upsilon(u,v,t)=t^{-\frac{h}{2q}}d_h v \left(1+\mathcal O(\Delta)+\mathcal  O(t^{-\frac{1}{2q}})\right), \quad \Delta=\sqrt{u^2+v^2}\to 0, \quad t\to\infty.
\end{gather*}

Consider a Lyapunov function candidate in the form
\begin{gather}\label{LFTh2}
\mathcal L(u,v,t)\equiv 
\begin{cases}
({\hbox{\rm sgn}}\, \eta)t^{\frac{1}{2q}} \mathcal H(u,v,t)+t^{-\frac{h-1}{2q}}\frac{d_h({\hbox{\rm sgn}}\, \eta)}{2}u v , & h>n=1,\\
({\hbox{\rm sgn}}\, \eta)t^{\frac{n}{2q}} \mathcal H(u,v,t)+t^{-\frac{h-1}{2q}} d_h ({\hbox{\rm sgn}}\, \eta)u v, & h>n>1,\\
({\hbox{\rm sgn}}\, \eta)t^{\frac{n}{2q}} \mathcal H(u,v,t) +t^{-\frac{h-1}{2q}} d_h\left\{({\hbox{\rm sgn}}\, \eta)u v+\frac{\lambda_n v^2}{2|\eta|}\right \}  , & h\leq n.
\end{cases}
\end{gather}
It can easily be checked that there exist $\Delta_1>0$ and $t_1\geq t_0$ such that $\mathcal L(u,v,t)$ satisfies the inequalities \eqref{Lnineq1}  for all $\Delta\leq \Delta_1$ and $t\geq t_1$ with some $L_\pm={\hbox{\rm const}}>0$. The total derivative of $\mathcal L(u,v,t)$ with respect to $t$ along the trajectories of system \eqref{uvsys1} is given by \eqref{dLuv1}, where
\begin{gather*}
	\mathcal  D_{1}(u,v,t)=
		\begin{cases}
			t^{-\frac{h}{2q}}\frac{d_h}{2}(|\eta|u^2+|\nu_1|v^2)(1 +\mathcal O(\Delta)+\mathcal O(t^{-\frac{1}{q}})), & n=1, \\
			t^{-\frac{h}{2q}}d_h(|\eta|u^2+|\nu_1|v^2)(1+\mathcal O(\Delta) +\mathcal O(t^{-\frac{1}{q}})), & n\neq 1
		\end{cases}
\end{gather*}
and $\mathcal  D_{2,N}(u,v,t) = \mathcal O(t^{-\frac{N-n+2}{2q}})\mathcal O(\Delta)$ as $\Delta\to 0$ and $t\to\infty$. It follows that there exist $N_0=n+h-1$, $\Delta_2\leq \Delta_1$ and $t_2\geq t_1$ such that 
$
\mathcal  D_{1}(u,v,t)\leq - t^{-\frac{h}{2q}} \gamma \Delta^2$, $\mathcal  D_{2}(u,v,t)\leq t^{-\frac{h+1}{2q}} C \Delta$
for all $N\geq N_0$ and $(u,v,t)\in\mathbb R^3$ such that $\Delta\leq \Delta_2$ and $t\geq t_2$, where $C={\hbox{\rm const}}>0$,  
and $\gamma_h=|d_h|/4$. By repeating the steps of the proof of Theorem~\ref{Th2}, we see that for all $\epsilon\in (0,\Delta_2)$ there exist $0<\delta_\epsilon<\epsilon$ and $t_\epsilon\geq t_2$ such that any solution of system \eqref{uvsys1} with initial data $\sqrt{u^2(t_\epsilon)+v^2(t_\epsilon)}\leq \delta$ and $\delta=\max\{\delta_\epsilon,\epsilon\sqrt{L_-/L_+}\}$  cannot exit from the domain $\{(u,v)\in\mathbb R^2: \Delta\leq \epsilon\}$ as $t\geq t_\epsilon$.  Returning to the original variables, we obtain the result of the Theorem.
\end{proof}

\begin{proof}[Proof of Theorem~\ref{Th24}]
Substituting $\varrho(t)=\varrho_{\ast,M}(t)+u(t)$, $\phi(t)=\phi_{\ast,M}(t)+v(t)$ into \eqref{rhopsi}, we obtain system \eqref{uvsys21}. 
In this case
\begin{gather*}
 Y_M(u,v,t)=v\left(\lambda_n t^{-\frac{n}{2q}}+\omega_m t^{-\frac{m}{2q}}\right) \left(1+\mathcal O(\Delta)+\mathcal O(t^{-\frac{1}{2q}})\right), \quad \Delta=\sqrt{u^2+v^2}\to 0, \quad t\to\infty.
\end{gather*}
Then, by repeating the proof of Theorem~\ref{Th21} with the Lyapunov function in the form \eqref{LFTh2}, with $H_M(u,v,t)$ instead of $\mathcal H(u,v,t)$, we obtain the result of the Theorem.
\end{proof}

\begin{proof}[Proof of Theorem~\ref{Th3}]  It follows from the first equation in \eqref{rhopsi} and assumption \eqref{asnzero} that for all $D>0$ there exist $t_1\geq t_0$ and $C_1>0$ such that 
$ |{d\rho}/{dt}|\geq t^{-\frac{n}{2q}}C_1$ for all $|\rho|\leq 4D$, $\psi\in\mathbb R$ and $t\geq t_1$. Integrating this inequality yields $|\rho(t)-\rho(t_1)|\geq C(t)> 0$ as $t> t_1$, where 
\begin{gather*}
C(t)\equiv \begin{cases}
	\frac{2q C_1}{2q-n} \left(t^{1-\frac{n}{2q}}-t_1^{1-\frac{n}{2q}}\right), & n<2q,\\
	C_1\left(\log t-\log t_1\right), & n=2q.
	\end{cases}
\end{gather*}
Hence, for all initial data $|\rho(t_1)|\leq D/2$ and $\psi(t_1)\in\mathbb R$ there exists $t_2\geq t_1$ such that $|\rho(t)|\geq D$ as $t\geq t_2$. Combining this with the second equation in \eqref{rhopsi}, we see that there exist $t_3\geq t_2$ and $C_2>0$ such that 
$|{d\psi}/{dt} |\geq t^{-\frac{1}{2q}}C_2$
for all $D\leq |\rho|\leq 2D$, $\psi\in\mathbb R$ and $t\geq t_3$. Then, by integration, we have
 \begin{gather*}
	|\psi(t)-\psi(t_3)|\geq \frac{2q C_2}{2q-1} \left(t^{1-\frac{1}{2q}}-t_3^{1-\frac{1}{2q}}\right), \quad t\geq t_3.
\end{gather*}
Therefore, for all initial data $D\leq |\rho(t_3)|\leq 3D/2$ and $|\psi(t_3)|\leq D/2$ there exists $t_4>t_3$ such that $|\rho(t)|\geq 2D$ and $|\psi(t)|\geq D$ as $t\geq t_4$.
\end{proof}

\section{Examples}\label{sEx}
In this section, we show how the proposed theory can be applied to examples of oscillatory systems with time-decaying perturbations. In particular, the conditions were obtained for the parameters of perturbations that guarantee the existence of a stable phase-locking regime with a resonant amplitude. The results are illustrated with numerical simulations. The last example analyzes the perturbed Duffing oscillator discussed in Section~\ref{sec1}.

\subsection{Example 1}
Consider the system
\begin{gather}\label{Ex1}
\frac{dr}{dt}=t^{-\frac{1}{2}} f_1(r,\varphi,S(t)), \quad \frac{d\varphi}{dt}=\omega(r)+t^{-\frac{1}{2}}g_1(r,\varphi,S(t))
\end{gather}
where  
\begin{gather*}
f_1(r,\varphi,S)  \equiv \beta(S) r \sin^2\varphi - \mu(S)\sin\varphi,\quad
g_1(r,\varphi,S)  \equiv \beta(S) \sin\varphi \cos\varphi - \frac{\mu(S)\cos\varphi}{r},\\
\omega(r)\equiv 1-\vartheta r^2, \quad 
\beta(S)\equiv \beta_0+\beta_1\sin S, \quad 
\mu(S)\equiv \mu_0+\mu_1 \sin S, \quad 
S(t) \equiv s_0 t + s_1 t^{\frac{1}{2}},
\end{gather*}
with constant parameters $s_k$, $\vartheta>0$, $\beta_k$ and $\mu_k$. We see that system \eqref{Ex1} has the form \eqref{PS} with $q=2$, $\mathcal R=\vartheta^{-1/2}$, $f(r,\varphi,S(t),t)\equiv t^{-1/2}f_1(r,\varphi,S(t))$ and $g(r,\varphi,S(t),t)\equiv t^{-1/2}g_1(r,\varphi,S(t))$. Note also that in the Cartesian coordinates $x=r\cos \varphi$, $y=-r\sin\varphi$ this system takes the form
\begin{align*}
&\frac{dx}{dt}=(1-\vartheta (x^2+y^2)) y, \\ 
&\frac{dy}{dt}=-(1-\vartheta (x^2+y^2))x + t^{-\frac{1}{2}}Z(x,y,S(t)),
\end{align*}
where $ Z(x,y,S)\equiv \mu(S)+\beta(S) y$.

1. Let $s_0=1/2$. Then, there exist $\kappa=\varkappa=1$, $a=(2\vartheta)^{-1/2}$ such that the resonance condition \eqref{rc} holds with $\eta=-\sqrt{2\vartheta}<0$. It can easily be checked that the change of variables described in Theorem~\ref{Th1} with $N=2$ transforms the system to
\begin{gather}\begin{split}\label{Ex1LO}
&\frac{d\rho}{dt}=t^{-\frac{1}{4}}\Lambda_1(\rho,\psi)+t^{-\frac{1}{2}}\Lambda_2(\rho,\psi)+\tilde \Lambda_2(\rho,\psi,S(t),t), \\
&\frac{d\psi}{dt}=t^{-\frac{1}{4}}\Omega_1(\rho,\psi)+t^{-\frac{1}{2}}\Omega_2(\rho,\psi)+\tilde\Omega_2(\rho,\psi,S(t),t),
\end{split}
\end{gather}
where
\begin{align*}
&\Lambda_1(\rho,\psi)\equiv \frac 12 \left(\frac{\beta_0}{\sqrt{2\vartheta}}-\mu_1 \cos\psi\right), &\quad 
&\Lambda_2(\rho,\psi)\equiv \frac{\beta_0 \rho}{2},\\
&\Omega_1(\rho,\psi)\equiv -\sqrt{2\vartheta} \rho, &\quad 
&\Omega_2(\rho,\psi)\equiv 
\frac12 \left(-2 \vartheta \rho^2 - s_1 +  \mu_1 \sqrt{2\vartheta}  \sin\psi\right),
\end{align*}
and $\tilde \Lambda_2(\rho,\psi,S,t)=\mathcal O(t^{-1})$, $\tilde \Omega_2(\rho,\psi,S,t)=\mathcal O(t^{-1})$ as $t\to\infty$ uniformly for all $|\rho|<\infty$, $(\psi,S)\in\mathbb R^2$. It is readily seen that assumption \eqref{asn} holds with $n=1$ and $m=2$.  

If $\mu_1\neq 0$ and $|\beta_0/\mu_1|<\sqrt{2\vartheta}$, then assumption \eqref{aszero} holds with
\begin{gather*}
\psi_0=\pm \theta_0 + {2\pi}k, \quad k\in\mathbb Z, \quad \nu_1=\pm\frac{\mu_1}{2}\sin \theta_0, \quad \theta_0=\arccos \left(\frac{\beta_0}{\sqrt{2\vartheta} \mu_1 }\right).
\end{gather*}
From Lemma~\ref{Lem01} it follows that if $\pm\mu_1<0$, then the equilibria $(0,\pm\theta_0({\hbox{\rm mod}} {2\pi}))$ in the corresponding limiting system are unstable. Hence, the associated regime is not realized in the full system.
Note that $d_{n,m}=\partial_\rho \Lambda_1(0,\psi_0)=0$. However, assumption \eqref{asst} holds with $h=2$ and $d_h=\beta_0$. It follows from Lemma~\ref{Lem2} and Theorem~\ref{Th23} that if $\pm\mu_1>0$ and $-|\mu_1|\sqrt{2\vartheta}<\beta_0<0$, then a stable phase locking regime with $r(t)\approx a$ and $\varphi(t)\approx S(t)\pm\theta_0({\hbox{\rm mod}} {2\pi})$ occurs in system \eqref{Ex1}. From Theorem~\ref{Th24} it follows that if $\pm\mu_1>0$ and $0<\beta_0<|\mu_1|\sqrt{2\vartheta}$, this regime is unstable. 

If $\mu_1\neq 0$, $|\beta_0/\mu_1|>\sqrt{2\vartheta}$ or $\mu_1=0$, $\beta_0\neq 0$, then assumption \eqref{asnzero} holds. It follows from Theorem~\ref{Th3} that, in this case, the asymptotic regime with $r(t)\approx a$ does not occur (see Fig.~\ref{FigEx11}).
\begin{figure}
\centering
{
   \includegraphics[width=0.4\linewidth]{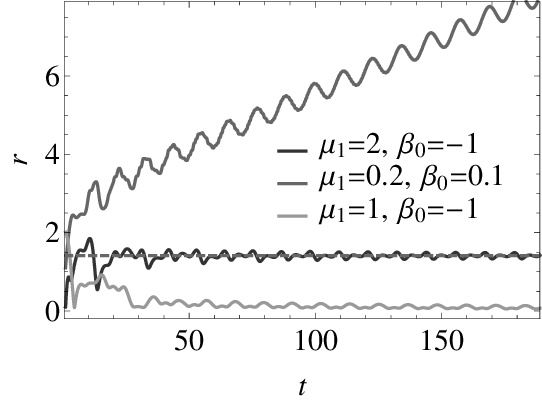}
}
\hspace{1ex}
{
   	\includegraphics[width=0.4\linewidth]{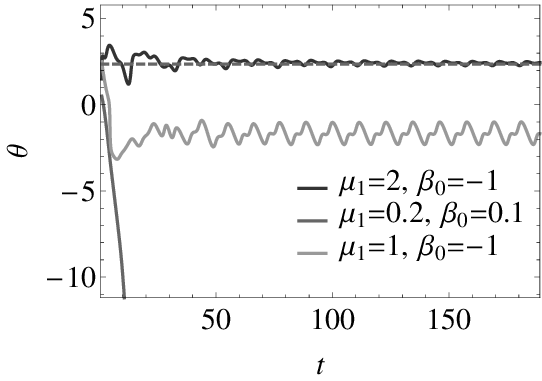}
}
\caption{\small The evolution of $r(t)$ and $\theta(t)\equiv \varphi(t)-S(t)$ for solutions to system \eqref{Ex1} with $s_0=1/2$, $s_1=1$, $\vartheta=1/4$, $\beta_1=1/2$, $\mu_0=-1/2$ with different values of the parameters $\beta_0$ and $\mu_1$. The dashed curves correspond to $r(t)\equiv a$ and $\theta(t)\equiv \theta_0$, where $a=\sqrt 2$ and $\theta_0=3\pi/4$.} \label{FigEx11}
\end{figure}

2. Let $s_0=1$. Then, there are $\kappa=1$, $\varkappa=2$, $a=(2\vartheta)^{-1/2}$ such that condition \eqref{rc} holds with $\eta=-\sqrt{2\vartheta}<0$. In this case, the transformation constructed in Theorem~\ref{Th1} with $N=2$ reduces system \eqref{Ex1} to \eqref{Ex1LO} with
\begin{align*}
&\Lambda_1(\rho,\psi)\equiv \frac{1}{\sqrt{8\vartheta}} \left(\beta_0+\frac{\beta_1}{2}\sin 2\psi\right), &\quad 
&\Lambda_2(\rho,\psi)\equiv  \frac{\rho}{2} \left(\beta_0+\frac{\beta_1}{2}\sin 2\psi\right),\\
&\Omega_1(\rho,\psi)\equiv -\sqrt{2\vartheta} \rho, &\quad 
&\Omega_2(\rho,\psi)\equiv 
\frac14 \left(-4 \vartheta \rho^2 - s_1 +  \beta_1 \cos 2\psi\right),
\end{align*}
and $\tilde \Lambda_2(\rho,\psi,S,t)=\mathcal O(t^{-1})$, $\tilde \Omega_2(\rho,\psi,S,t)=\mathcal O(t^{-1})$ as $t\to\infty$ uniformly for all $|\rho|<\infty$, $(\psi,S)\in\mathbb R^2$. We see that assumption \eqref{asn} holds with $n=1$ and $m=2$.  

If $\beta_1\neq 0$, $ |\beta_0/\beta_1|<1/2$, then the system satisfies \eqref{aszero} with 
\begin{gather*}
\psi_0=(-1)^k  \theta_0 + \frac{\pi k}{2}, \quad 
\nu_1= (-1)^k\frac{\beta_1}{\sqrt{8\vartheta}}\cos2\theta_0, \quad 
k\in\mathbb Z, \quad 
\theta_0=\frac{1}{2}\arcsin \left(-\frac{2\beta_0}{\beta_1}\right).
\end{gather*}
It follows from Lemma~\ref{Lem01} that if $(-1)^k\beta_1<0$, then the equilibria $(0, (-1)^k\theta_0+\pi k/2)$, $k\in\mathbb Z$ in the limiting system and the corresponding regime in the full system are unstable. Since $d_{n,m}=\partial_\rho \Lambda_1(0,\psi_0)=0$ and $\partial_\rho \Lambda_2(0,\psi_0)+\partial_\psi\Omega_2(0,\psi_0)=\beta_0$, we see that assumption \eqref{asst} holds with $h=2$ and $d_h=\beta_0$. If $(-1)^k\beta_1>0$ and $-|\beta_1|/2<\beta_0<0$, then it follows from Lemma~\ref{Lem2} and Theorem~\ref{Th23} that a stable phase locking occurs in the system such that $r(t)\approx a$ and $\varphi(t)\approx S(t)/2+\psi_0$. From Theorem~\ref{Th24} it follows that if $(-1)^k\beta_1>0$ and $0<\beta_0<|\beta_1|/2$, this regime is unstable. 

If $\beta_1\neq 0$, $|\beta_0/\beta_1|>1/2$ or $\beta_1=0$, $\beta_0\neq 0$, then it follows from Theorem~\ref{Th3} that the asymptotic regime with $r(t)\approx a$ does not occur (see Fig.~\ref{FigEx12}).

\begin{figure}
\centering
{
   \includegraphics[width=0.4\linewidth]{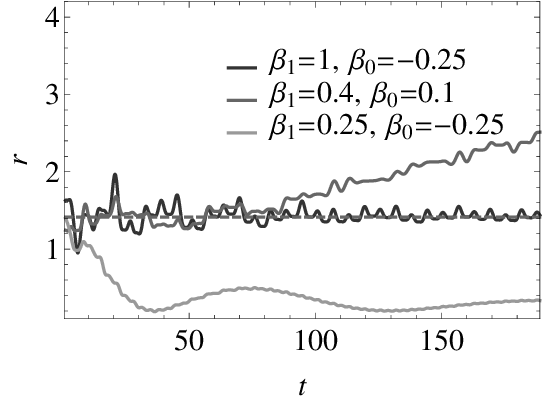}
}
\hspace{1ex}
{
   	\includegraphics[width=0.4\linewidth]{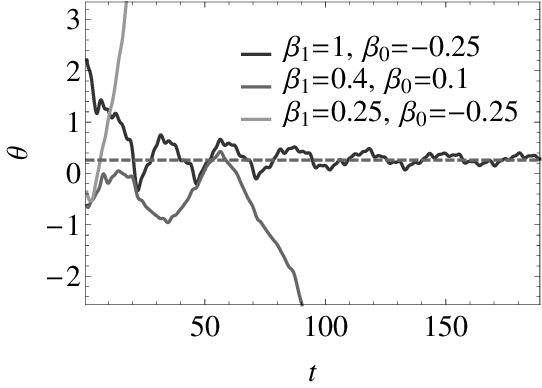}
}
\caption{\small The evolution of $r(t)$ and $\theta(t)\equiv \varphi(t)-S(t)/2$ for solutions to system \eqref{Ex1} with $s_0=1$, $s_1=\vartheta=\mu_1=1/4$, $\mu_0=0$ with different values of the parameters $\beta_0$ and $\beta_1$. The dashed curves correspond to $r(t)\equiv a$ and $\theta(t)\equiv \theta_0$, where $a=\sqrt 2$ and $\theta_0=\pi/12$.} \label{FigEx12}
\end{figure}

3. Finally, let $s_0=1/4$. Then, there exist $\kappa=2$, $\varkappa=1$ and $a=(2\vartheta)^{-1/2}$ such that the resonance condition holds with $\eta=-\sqrt{2\vartheta}<0$. Note that the transformation described in Theorem~\ref{Th1} with $N=1$ reduces system \eqref{Ex1} to 
\begin{align*}
&\frac{d\rho}{dt}=t^{-\frac{1}{4}}\Lambda_1(\rho,\psi)+\tilde \Lambda_1(\rho,\psi,S(t),t), \\
&\frac{d\psi}{dt}=t^{-\frac{1}{4}}\Omega_1(\rho,\psi)+\tilde\Omega_1(\rho,\psi,S(t),t),
\end{align*}
with $\Lambda_1(\rho,\psi)\equiv {a\beta_0}/{2}$, $\Omega_1(\rho,\psi)\equiv -\sqrt{2\vartheta} \rho$, and $\tilde \Lambda_1(\rho,\psi,S,t)=\mathcal O(t^{-1})$, $\tilde \Omega_1(\rho,\psi,S,t)=\mathcal O(t^{-1})$ as $t\to\infty$ uniformly for all $|\rho|<\infty$, $(\psi,S)\in\mathbb R^2$.
It follows from Theorem~\ref{Th3} that if $\beta_0\neq 0$, the asymptotic regime with $r(t)\approx a$ does not occur. In this case, the behaviour of system \eqref{Ex1} is qualitatively independent of the oscillatory part of the perturbations. 

\subsection{Example 2} Consider the following system:
\begin{gather}\label{Ex2}
\begin{split}
&\frac{dr}{dt}=t^{-\frac{1}{2}} f_1(r,\varphi,S(t))+t^{-1} f_2(r,\varphi,S(t)), \\ 
&\frac{d\varphi}{dt}=\omega(r)+t^{-\frac{1}{2}}g_1(r,\varphi,S(t))+t^{-1}g_2(r,\varphi,S(t))
\end{split}
\end{gather}
where  
\begin{align*}
&	f_1(r,\varphi,S)  \equiv -\alpha(S) r^3 \sin\varphi \cos^3\varphi,\quad
&&	f_2(r,\varphi,S)  \equiv \beta(S) r \sin^2\varphi,\\  
&	g_1(r,\varphi,S)  \equiv -\alpha(S) r^2 \cos^4\varphi, \quad 
&&	g_2(r,\varphi,S)  \equiv \frac{\beta(S)}{2}  \sin 2\varphi,\\
&	\alpha(S)\equiv \alpha_0+\alpha_1\sin S,\quad
&&	\beta(S)\equiv \beta_0+\beta_1\sin S,\\
&	\omega(r)\equiv 1-\vartheta r^2, \quad 
&&	S(t) \equiv s_0 t + s_1 t^{\frac{1}{2}}+ s_2 \log t
\end{align*}
with constant parameters $s_k$, $\vartheta>0$, $\alpha_k$, $\beta_k$, $\alpha_1\neq 0$ It can be easily seen that system \eqref{Ex2} has the form \eqref{PS} with $q=2$, $\mathcal R=\vartheta^{-1/2}$, $f(r,\varphi,S(t),t)\equiv t^{-1/2}f_1(r,\varphi,S(t))+t^{-1}f_2(r,\varphi,S(t))$ and $g(r,\varphi,S(t),t)\equiv t^{-1/2}g_1(r,\varphi,S(t))+t^{-1}g_2(r,\varphi,S(t))$. In the Cartesian coordinates $x=r\cos \varphi$, $y=-r\sin\varphi$ this system takes the form
\begin{align*}
&\frac{dx}{dt}=(1-\vartheta (x^2+y^2)) y, \\
& \frac{dy}{dt}=-(1-\vartheta (x^2+y^2))x + t^{-\frac{1}{2}}\alpha(S(t))x^3+ t^{-1}\beta(S(t))y.
\end{align*}
Let $s_0=1/2$. Then, there exist $\kappa=1$, $\varkappa=2$, $a=(2\vartheta)^{-1/2}$ such that the resonance condition \eqref{rc} holds with $\eta=-\sqrt{2\vartheta}<0$. It can easily be checked that the change of variables described in Theorem~\ref{Th1} with $N=4$ transforms the system to
\begin{gather}\begin{split}\label{Ex2LO}
&\frac{d\rho}{dt}=\sum_{i=1}^4 t^{-\frac{i}{4}}\Lambda_i(\rho,\psi)+\tilde \Lambda_4(\rho,\psi,S(t),t), \\
&\frac{d\psi}{dt}=\sum_{i=1}^4 t^{-\frac{i}{4}}\Omega_i(\rho,\psi)+\tilde\Omega_4(\rho,\psi,S(t),t),
\end{split}
\end{gather}
where
\begin{align*}
 \Lambda_1(\rho,\psi) \equiv& -\frac{a^3 \alpha_1}{8} \cos 2 \psi, \\
 \Lambda_2(\rho,\psi) \equiv &-\frac{3a^2 \alpha_1 \rho}{8}  \cos 2 \psi,\\
 \Lambda_3(\rho,\psi) \equiv& 
 \frac{a }{32}\left(16 \beta_0 - 
   \alpha_1 (12 a^4 \alpha_0 + 12 \rho^2 + 5 a^6 \alpha_0 \vartheta) \cos
     2 \psi + 8 \beta_1 \sin 2 \psi + 3 a^4 \alpha_1^2 \sin 4 \psi\right), \\
  \Lambda_4(\rho,\psi)\equiv &
-\frac{\rho}{64} (\alpha_1 (111 a^4 \alpha_0 + 8 \rho^2 + 
      336 a^6 \alpha_0 \vartheta + 432 a^8 \alpha_0 \vartheta^2) \cos 
     2 \psi \\
		& - 
   2 (8 + 16 \beta_0 + 8 \beta_1 \sin 2 \psi + 
      3 a^4 \alpha_1^2 (5 + 16 a^2 \vartheta) \sin 4 \psi)), \\
	\Omega_1(\rho,\psi) \equiv & -\sqrt{2\vartheta} \rho, \\
 \Omega_2(\rho,\psi) \equiv &
\frac{1}{8} 
\left(-3 a^2 \alpha_0 - 2 s_1 - 8 \rho^2 \vartheta + 4 a^2 \alpha_1 \cos \psi \sin \psi\right),\\
  \Omega_3(\rho,\psi) \equiv &
\frac{a\rho}{4} (-3 \alpha_0 + 2 \alpha_1 \sin 2 \psi), \\
  \Omega_4(\rho,\psi) \equiv &
\frac{1}{3456}(-54 (a^4 (57 \alpha_0^2 + 8 \alpha_1^2) + 24 \alpha_0 \rho^2 + 
     32 s_2) - 3 a^6 (3537 \alpha_0^2 + 437 \alpha_1^2) \vartheta \\
		& - 
  16 a^8 (918 \alpha_0^2 + 139 \alpha_1^2) \vartheta^2 + 
  864 \beta_1 \cos 2 \psi \\
	& + 
  54 \alpha_1 (3 a^4 \alpha_1 (3 + 8 a^2 \vartheta) \cos 
       4 \psi + (16 \rho^2 + 
        a^4 \alpha_0 (67 + a^2 \vartheta (173 + 216 a^2 \vartheta))) \sin
       2 \psi))
\end{align*}
and $\tilde \Lambda_4(\rho,\psi,S,t)=\mathcal O(t^{-5/4})$, $\tilde \Omega_4(\rho,\psi,S,t)=\mathcal O(t^{-5/4})$ as $t\to\infty$ uniformly for all $|\rho|<\infty$, $(\psi,S)\in\mathbb R^2$. It is readily seen that assumption \eqref{asn} holds with $n=1$ and $m=2$. 

Note that system \eqref{Ex2LO} satisfies \eqref{aszero} with 
\begin{gather*}
\psi_0=\frac{\pi}{4} + \frac{\pi k}{2}, \quad 
\nu_1= (-1)^k\frac{a^3 \alpha_1}{4}, \quad 
k\in\mathbb Z.
\end{gather*}
Since $\eta<0$, it follows from Lemma~\ref{Lem01} that if $(-1)^k \alpha_1<0$, then the equilibrium $(0, \pi/4+\pi k/2)$ is unstable in the limiting system for all $k\in\mathbb Z$. Hence, the corresponding resonant regimes do not occur in the full system. Moreover, we see that $d_{n,m}=\partial_\rho \Lambda_1(0,\psi_0)=0$, $\partial_\rho \Lambda_i(0,\psi_0)+\partial_\psi\Omega_i(0,\psi_0)=0$ for $1\leq i\leq 3$, and the assumption \eqref{asst} holds with $h=4$ and 
\begin{gather*}
d_h=\frac{1+2\beta_0+(-1)^{k+1}\beta_1}{4}
\end{gather*}
Thus, if $(-1)^k\alpha_1>0$ and $(-1)^k\beta_1>1+2\beta_0$, it follows from Lemma~\ref{Lem2} and Theorem~\ref{Th23} that a stable phase locking occurs in the system such that $r(t)\approx a$ and $\varphi(t)\approx S(t)/2+\psi_0$ (see Fig.~\ref{FigEx21}).

\begin{figure}
\centering
{
   \includegraphics[width=0.4\linewidth]{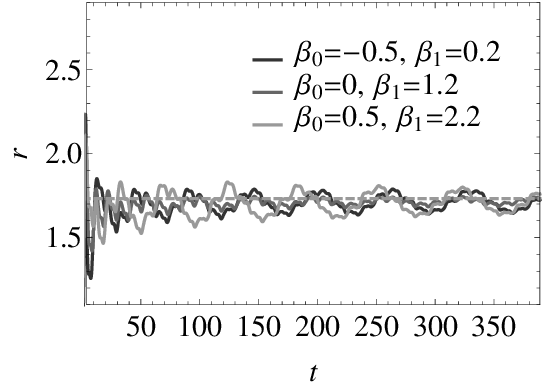}
}
\hspace{1ex}
{
   	\includegraphics[width=0.4\linewidth]{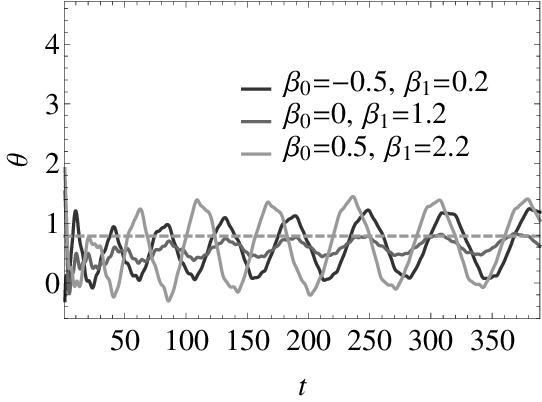}
}
\caption{\small The evolution of $r(t)$ and $\theta(t)\equiv \varphi(t)-S(t)/2$ for solutions to system \eqref{Ex2} with $s_0=1/2$, $s_1=1$, $s_2=0$, $\vartheta=1/4$, $\alpha_0=0.1$, $\alpha_1=0.15$ with different values of the parameters $\beta_0$ and $\beta_1$. The dashed curves correspond to $r(t)\equiv a$ and $\theta(t)\equiv \pi/4$, where $a=\sqrt 2$.} \label{FigEx21}
\end{figure}

\subsection{Example 3} Finally, consider again equation \eqref{Ex0}. It was shown in Section~\ref{sec1} that this system correspond to \eqref{PS} with $q=2$, $s_0=3/2$, and functions $\omega(r)$, $f(r,\varphi,S,t)$, $g(r,\varphi,S,t)$ defined by \eqref{omegaeq} and \eqref{fgex0}. Note that $0<\omega(r)<1$ for all $0<|r|<(2\vartheta)^{-1/2}$ and $\omega(r)=1-3 \vartheta r^2/8-35\vartheta^2 r^4/256+\mathcal O(\vartheta^4)$ as $\vartheta\to 0$. Hence, there exist $\kappa$, $\varkappa\in\mathbb Z_+$ and $0<a<(2\vartheta)^{-1/2}$ such that the condition \eqref{rc} holds with $\eta<0$.

Let $\kappa=1$ and $\varkappa=2$. Then, the transformations \eqref{ch1}, \eqref{ch2} with $N=2$ reduce the system to \eqref{Ex1LO} with
\begin{align*}
&\Lambda_1(\rho,\psi)= \frac{a}{4} (2 \beta_0 +\delta_1 \sin (2\psi-\sigma))+\mathcal O(\vartheta), & \quad 
&\Lambda_2(\rho,\psi)\equiv \frac{\rho}{4} \left(2 \beta_0+\delta_1 \sin (2\psi-\sigma) +\mathcal O(\vartheta)\right),\\
&\Omega_1(\rho,\psi)\equiv  \eta \rho, &\quad 
&\Omega_2(\rho,\psi)\equiv 
\frac 14 (-  2\alpha_0 +  \delta_1 \cos (2 \psi-\sigma))+\mathcal O(\vartheta),
\end{align*}
as $\vartheta\to 0$ and $\tilde \Lambda_2(\rho,\psi,S,t)=\mathcal O(t^{-1})$, $\tilde \Omega_2(\rho,\psi,S,t)=\mathcal O(t^{-1})$ as $t\to\infty$ uniformly for all $|\rho|<\infty$, $(\psi,S)\in\mathbb R^2$, where $\delta_1=\sqrt{\alpha_1^2+\beta_1^2}$ and $\sigma=\arcsin (\alpha_1/\delta_1)$. We see that assumption \eqref{asn} holds with $n=1$ and $m=2$.  

If $\delta_1\neq 0$ and $|\beta_0|<\delta_1/2$, then the system satisfies \eqref{aszero} with 
\begin{gather*}
\psi_0=(-1)^j  \theta_0 + \frac{\sigma+\pi j}{2}+\mathcal O(\vartheta), \quad 
\nu_1= (-1)^j\frac{a \delta_1}{2}\cos2\theta_0+\mathcal O(\vartheta), \quad 
j\in\mathbb Z, \quad 
\theta_0=\frac{1}{2}\arcsin \left(-\frac{2\beta_0}{\delta_1}\right).
\end{gather*}
It follows from Lemma~\ref{Lem01} that the equilibria $(0, (\sigma+\pi)/2-\theta_0+ \pi j)$, $j\in\mathbb Z$ in the limiting system and the corresponding regime in the full system are unstable. Since $d_{n,m}=\partial_\rho \Lambda_1(0,\psi_0)=0$ and $\partial_\rho \Lambda_2(0,\psi_0)+\partial_\psi\Omega_2(0,\psi_0)=\beta_0$, we see that assumption \eqref{asst} holds with $h=2$ and $d_h=\beta_0+\mathcal O(\vartheta)$ as $\vartheta\to 0$. If $-\delta_1/2<\beta_0<0$, then it follows from Lemma~\ref{Lem2} and Theorem~\ref{Th23} that a stable phase locking occurs in the system such that $r(t)\approx a$ and $\varphi(t)\approx S(t)/2+\theta_0+\sigma/2+\pi j$, $j\in\mathbb Z$. From Theorem~\ref{Th24} it follows that if $0<\beta_0<\delta_1/2$, this regime is unstable. 

It follows from Theorem~\ref{Th3} that if $\delta_1=0$, $\beta_0\neq 0$ or $\delta_1\neq 0$, $|\beta_0|>\delta_1/2$, then the asymptotic regime with $r(t)\approx a$ does not occur. 

Note that the root of the equation $\omega(a)=\kappa s_0/\varkappa$ can be found numerically. In particular, if $\vartheta=1/4$, we have $a\approx 1.27$ (see Fig.~\ref{figomega} and Fig.~\ref{FigEx0}, c). 
\begin{figure}
\centering
{
   \includegraphics[width=0.4\linewidth]{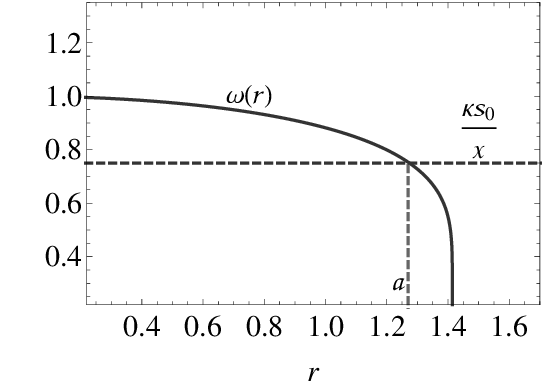}
}
\caption{\small Finding $a$, when $\vartheta=1/4$, $\kappa=1$ and $\varkappa=2$.} \label{figomega}
\end{figure}

\section{Conclusion}
Thus, the resonant effect of damped oscillatory perturbations on non-isochronous systems has been investigated. In particular, we have deduced the model non-autonomous system \eqref{trsys}, which describes the approximate average dynamics. It turned out that this system is similar to the pendulum-type equations with additional terms decaying in time. Indeed, the truncated limiting system \eqref{limsys} can be written as
\begin{gather*}
\frac{d^2\hat\phi}{d\tau^2}-\tau^{-\frac{n-1}{2q}}\eta_n \Lambda_n\left(\eta^{-1}\frac{d\hat\phi}{d\tau},\hat\phi\right)=0, \quad \eta_n=\eta \left(\frac{2q}{2q-1}\right)^{\frac{n-1}{2q-1}}, \quad \tau=\left(\frac{2q}{2q-1}\right)t^{1-\frac{1}{2q}},
\end{gather*}
where $\Lambda_n(\rho,\psi)$ is $2\pi$-periodic with respect to $\psi$. In this case, the additional terms in the model system depends on the perturbations of the oscillatory system. Note that similar but autonomous equations arise in the theory of nonlinear resonance when considering perturbations with a small parameter.~\cite{BVC79,SUZ88}. The study of the structure of the model system has led to conditions that guarantee the existence of the phase-locking regime with a resonant amplitude. Violation of these conditions can lead to significant phase mismatch and the absence of a corresponding resonant mode. The proposed method is based on long-term asymptotic analysis of the model system and the proof of the stability of the corresponding solutions in the full system using Lyapunov function technique. We have shown that time-decaying perturbations can be used to control the dynamics of nonlinear systems. For example, the perturbation parameters can be chosen to ensure the appearance of near-periodic solutions with a given resonant amplitude. 

Note also that perturbations of isochronous systems have not been discussed here. In this case, the proposed theory cannot be applied directly due to different form of the model systems. Multi-frequency systems, where the problem of small denominators may arise, have also not been considered in the paper. These problems deserve special attention and will be discussed elsewhere.

%\section*{Acknowledgements}
%The research is supported by the Russian Science Foundation (Grant No. 23-11-00009).

}

\begin{thebibliography}{99}
\bibitem{GH83} J. Guckenheimer, P. Holmes, \textit{Nonlinear Oscillations, Dynamical Systems and Bifurcations of Vector Fields},  Springer, New York,  1983.

\bibitem{AF06} A. Fidlin, \textit{Nonlinear Oscillations in Mechanical Engineering}, Springer, Berlin, Heidelberg, New York, 2006.

\bibitem {CCT95} C. Castillo-Ch\'{a}vez, H.R. Thieme, \textit{Asymptotically autonomous epidemic models}. In: O. Arino, D. Axelrod, M. Kimmel, M. Langlais (Eds.), Mathematical Population Dynamics: Analysis of Heterogenity, Theory of Epidemics, vol. 1, Wuertz, 1995, p. 33--50.

\bibitem{DS22} D. Scarcella, \textit{Weakly asymptotically quasiperiodic solutions for time-dependent Hamiltonians with a view to celestial mechanics}, 2022, arXiv: 2211.06768.

\bibitem{BG08} A. D. Bruno, I. V. Goryuchkina, \textit{Boutroux asymptotic forms of solutions to Painlev\'{e} equations and power geometry}, Doklady Math., 78 (2008), 681--685.

\bibitem{KF13} V. V. Kozlov, S. D. Furta, \textit{Asymptotic Solutions of Strongly Nonlinear Systems of Differential Equations}, Springer, New York, 2013.
 
\bibitem{SFR19} S. F. Rohmah et al, \textit{Time-dependent damping effect for the dynamics of DNA transcription}, J. Phys.: Conf. Ser., 1204 (2019), 012012.

\bibitem{JM23} S. Ji, M. Mei, \textit{Optimal decay rates of the compressible Euler equations with time-dependent damping in $\mathbb R^n$: (I) under-damping case}, J Nonlinear Sci., 33 (2023), 7.

\bibitem{LM56} L. Markus, \textit{ Asymptotically autonomous differential systems}. In: S. Lefschetz (ed.), Contributions to the theory of nonlinear oscillations III, Ann. Math. Stud., vol. 36, pp. 17--29, Princeton University Press, Princeton, 1956.

\bibitem{LRS02} J. A. Langa, J. C. Robinson, A. Su\'{a}rez, \textit{Stability, instability and bifurcation phenomena in nonautonomous differential equations}, Nonlinearity, 15 (2002), 887--903.

\bibitem{KS05} P. E. Kloeden, S. Siegmund, \textit{Bifurcations and continuous transitions of attractors in autonomous and nonautonomous systems}, Internat. J. Bifur. Chaos., 15 (2005), 743--762.

\bibitem{MR08} M. Rasmussen, \textit{Bifurcations of asymptotically autonomous differential equations}, 
Set-Valued Anal., 16 (2008), 821--849.

\bibitem{OS22Non} O. A. Sultanov, \textit{Stability and bifurcation phenomena in asymptotically Hamiltonian systems}, Nonlinearity, 35 (2022), 2513--2534.

\bibitem{LDP74} L. D. Pustyl'nikov, \textit{Stable and oscillating motions in nonautonomous dynamical systems. A generalization of C. L. Siegel's theorem to the nonautonomous case}, Math. USSR-Sbornik, 23 (1974), 382--404.

\bibitem{HT94} H. Thieme, \textit{Asymptotically autonomous differential equations in the plane}, 
Rocky Mountain J. Math., 24 (1994), 351--380.

\bibitem{OS21IJBC} O. A. Sultanov, \textit{Damped perturbations of systems with center-saddle bifurcation}, Internat. J. Bifur. Chaos., 31 (2021), 2150137.

\bibitem{HL75} W. A. Harris and D. A. Lutz, \textit{Asymptotic integration of adiabatic oscillators}, 
J. Math. Anal. Appl., 51 (1975), 76--93.

\bibitem{MP85} M. Pinto, \textit{Asymptotic integration of second-order linear differential equations}, J. Math. Anal. Appl., 111 (1985), 388--406.

\bibitem{PN07} P. N. Nesterov , \textit{Averaging method in the asymptotic integration problem for systems with oscillatory-decreasing coefficients}, Differ. Equ. 43 (2007), 745--756.

\bibitem{BN10} V. Burd, P. Nesterov, \textit{Parametric resonance in adiabatic oscillators}, Results. Math., 58 (2010), 1--15.

\bibitem{OS23JMS} O. A. Sultanov, \textit{Asymptotic analysis of systems with damped oscillatory perturbations}, J. Math. Sci. 269 (2023), 111--128.

\bibitem{OS21DCDS} O. A. Sultanov, \textit{Bifurcations in asymptotically autonomous Hamiltonian systems under oscillatory perturbations}, Discrete \& Continuous Dynamical Systems, 41 (2021), 5943--5978.

\bibitem{OS21JMS} O. A. Sultanov, \textit{Decaying oscillatory perturbations of Hamiltonian systems in the plane}, Journal of Mathematical Sciences, 257 (2021), 705--719.

\bibitem{BVC79} B. V. Chirikov, \textit{A universal instability of many-dimensional oscillator systems}, Physics Reports, 52 (1979), 263--379.

\bibitem{SUZ88} R. Z. Sagdeev, D.A. Usikov, G.M. Zaslavsky, \textit{Nonlinear Physics: From the Pendulum to Turbulence and Chaos}, Harwood Academic Publishers, New York, 1988.

\bibitem{Sos97} S. M. Soskin, D. G. Luchinsky, R. Mannella, A. B. Neiman, and P. V. E. McClintock, \textit{Zero-dispersion nonlinear resonance}, Internat. J. Bifur. Chaos., 7 (1997), 923--936.

\bibitem{JCRS11} S. Jeyakumari, V. Chinnathambi, S. Rajasekar, and M. A. F. Sanjuan, \textit{Vibrational resonance in an asymmetric Duffing oscillator}, Internat. J. Bifur. Chaos., 21 (2011), 275--286.

\bibitem{BM61} N. N. Bogolubov, Yu. A. Mitropolsky, \textit{Asymptotic Methods in Theory of Non-linear Oscillations}, Gordon and Breach, New York, 1961.

\bibitem{AKN06} V. I. Arnold, V. V. Kozlov, A. I. Neishtadt, \textit{Mathematical Aspects of Classical and Celestial Mechanics}, Springer, Berlin, 2006.

\bibitem{LK15} L. A. Kalyakin, \textit{Lyapunov functions in theorems of justification of asymptotics}, Mat. Notes, 98 (2015), 752--764.

\bibitem{Halil} H. K. Khalil, \textit{Nonlinear Systems}, Prentice Hall, Upper Saddle River, New Jersey, 2002.

\end{thebibliography}
\end{document}